\pdfoutput=1
\documentclass{amsart}
\usepackage{a4wide}
\usepackage{graphicx,color} 
\usepackage{amsmath,amssymb,amsfonts,amsthm}
\usepackage{mathtools}
\usepackage{tikz}
\usetikzlibrary{arrows.meta}
\usepackage{xcolor}

\definecolor{white}{rgb}{1,1,1}
\definecolor{Black}{rgb}{0,0,0}

\newtheorem{thm}{Theorem}[section]
\newtheorem{defn}[thm]{Definition}
\newtheorem{rem}[thm]{Remark}
\newtheorem{ex}[thm]{Example}
\newtheorem{lem}[thm]{Lemma}
\newtheorem{prop}[thm]{Proposition}

\newtheorem{cor}[thm]{Corollary}

\newtheorem{prob}[thm]{Problem}

\newcommand{\ap}{\includegraphics[width=0.08\linewidth]{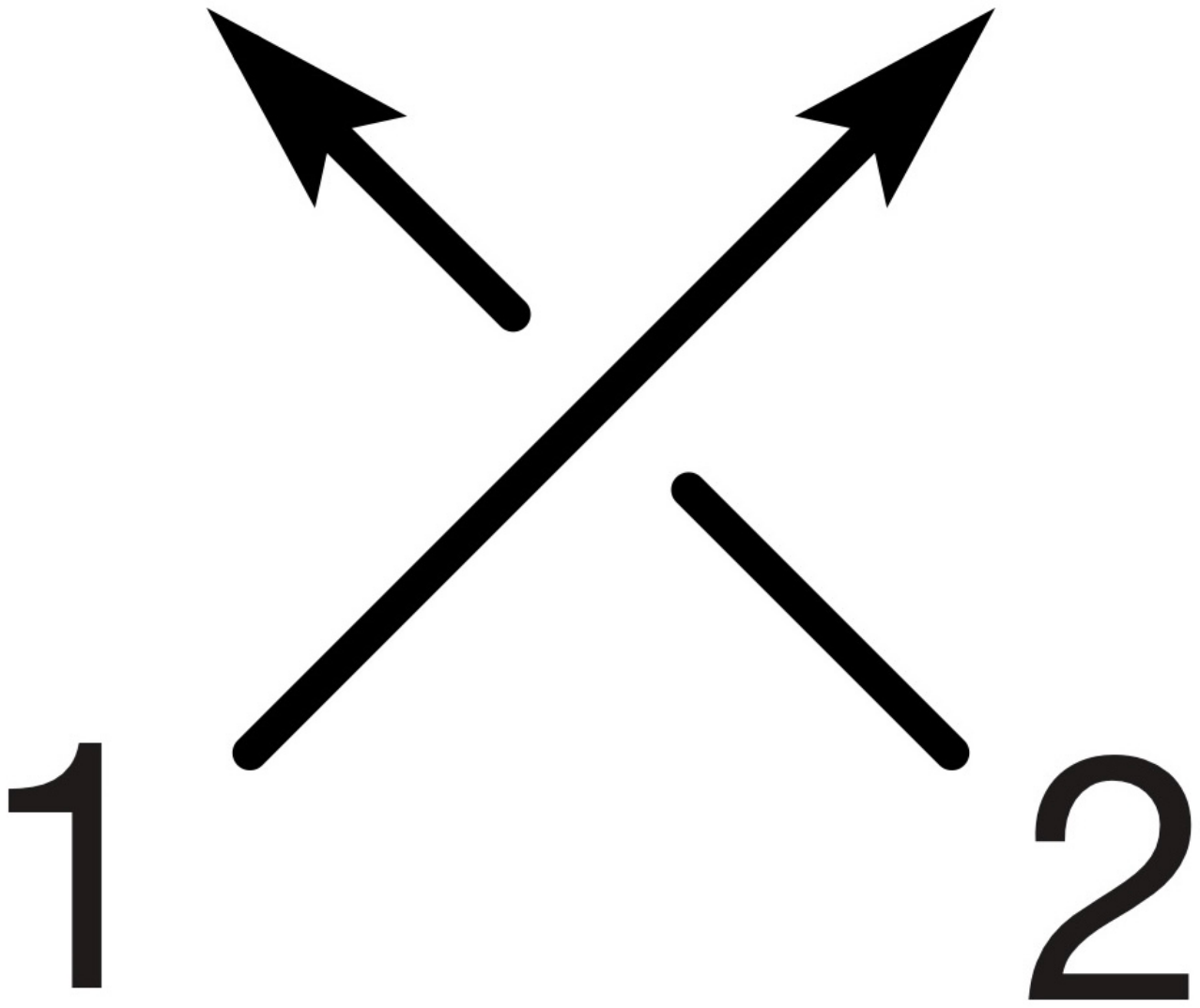}}
\newcommand{\bm}{\includegraphics[width=0.08\linewidth]{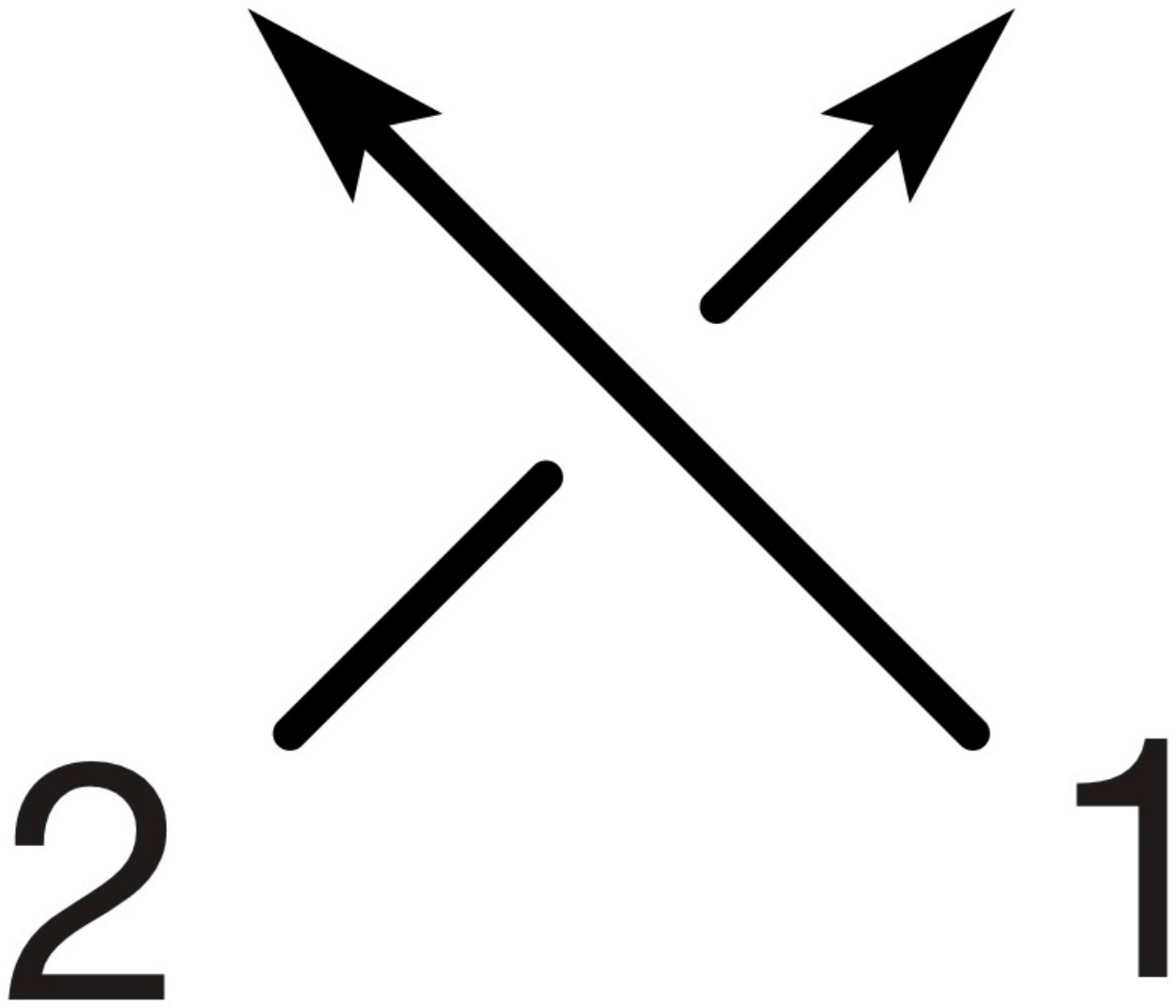}}

\newcommand{\R}{\mathcal{R}}
\newcommand{\II}{I\!I}
\newcommand{\III}{I\!I\!I}
\newcommand{\tI}{\textup{type~I}}
\newcommand{\tII}{\textup{type~I\!I}}
\newcommand{\tIII}{\textup{type~I\!I\!I}}
\newcommand{\oneA}{\textup{1a}}
\newcommand{\oneB}{\textup{1b}}
\newcommand{\oneC}{\textup{1c}}
\newcommand{\oneD}{\textup{1d}}
\newcommand{\oneAS}{\operatorname{1\ast}}

\newcommand{\oneSH}{\operatorname{1\sharp}}
\newcommand{\twoA}{\textup{2a}}
\newcommand{\twoB}{\textup{2b}}
\newcommand{\twoC}{\textup{2c}}
\newcommand{\twoD}{\textup{2d}}
\newcommand{\threeA}{\textup{3a}}
\newcommand{\threeB}{\textup{3b}}
\newcommand{\threeC}{\textup{3c}}
\newcommand{\threeD}{\textup{3d}}
\newcommand{\threeE}{\textup{3e}}
\newcommand{\threeF}{\textup{3f}}
\newcommand{\threeG}{\textup{3g}}
\newcommand{\threeH}{\textup{3h}}
\newcommand{\threeAS}{\operatorname{3\ast}}

\title{Minimal Generating sets of Reidemeister moves}
\vspace{-2mm}
\author{Noboru Ito}
\address{Department of Mathematics, Faculty of Engineering, Shinshu University, 4-17-1 Wakasato, Nagano 380-8553, Japan}
\email{nito@shinshu-u.ac.jp}
\author{Yuichiro Iwamoto}
\address{Department of Science and Technology, Graduate School of Medicine, Science and Technology,
Shinshu University, 3-1-1 Asahi, Matsumoto, Nagano 390-8626, Japan}
\email{IwamotoY.math@gmail.com}
\date{November 25,2025.}

\begin{document}
%\title{Generating sets of Reidemeister moves}
\maketitle
%\vspace{-1cm}
\begin{abstract}
We determine whether each known generating set of arbitrary oriented Reidemeister moves is  minimal. We then provide a complete classification of minimal generating sets that include a  coherent Reidemeister move of $\tII$. 
We also classify all minimal generating sets that include a braid-type Reidemeister move of $\tIII$.  Beyond these two cases, we identify $16$ possible candidates for minimal generating sets.  Among them, we prove that $12$ are indeed minimal, whereas the minimality and even the generating property of the remaining $4$ sets remains unsolved (cf.~Remark~\ref{remhomotopy}).  
%If the set contains a single $3$-move and if is of braid-type, we determine the minimal generating sets.  
%show that at least five Reidemeister moves are required for any choice of the third Reidemeister move is of  braid-type; in this case, the possible choices reduced to one variant of the third move (six possibilities), two variants of the second moves (four  possibilities), and two variants of the first moves (four  possibilities).   For all $144$ such combinations, we determine which  generate ambient isotopy and, among them, which are minimal.          
\end{abstract}
%\vspace{-0.7cm}
\section{Introduction}\label{Intro}
%\vspace{-0.3cm}
Reidemeister moves are local replacements of knot, link, or tangle diagrams, which are generic plane curves with transverse double points and over/under  information.  Two diagrams represent the same oriented link/tangle if they are related by a finite sequence of local deformations, each supported inside an oriented embedded disk in the plane, called a \emph{changing disk}: 
\vspace{-0.3cm}
\begin{figure}[h]
    \centering
    \includegraphics[width=0.8\linewidth]{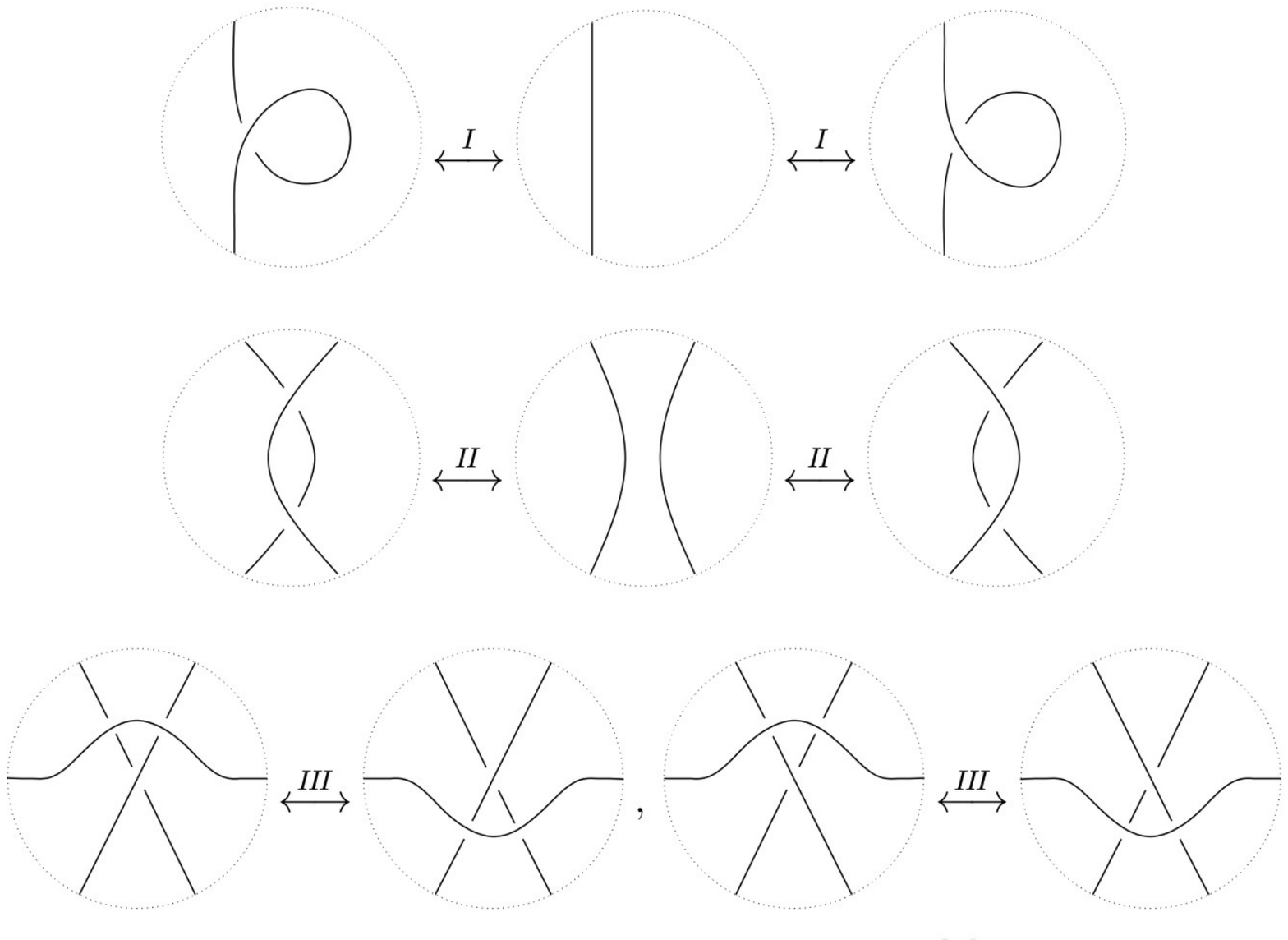}
    \label{UnoriReide}
\end{figure}
\vspace{-0.5cm}
%\begin{align*}
%  \Iun  \overset{I}   {\longleftrightarrow} \Imid \overset{I}  {\longleftrightarrow}   \Io  
%\end{align*}

%\begin{align*}
%  \IIu  \overset{\II} {\longleftrightarrow} \IIn  \overset{\II}{\longleftrightarrow}   \IIo 
%\end{align*}

%\begin{align*}
%  \IIIu   \overset{\III}{\longleftrightarrow} \IIIo \; ,  \;
%  \IIIucc \overset{\III}{\longleftrightarrow} \IIIocc
%\end{align*}

%\vspace{-0.1cm}
%One of the authors, NI, has often been asked about the generating sets of oriented Reidemeister moves. 
Questions of oriented generating sets have  long circulated.   
%Despite the foundational nature of the question, it had not been clearly formulated in the literature until \cite{Polyak2010}, which studied oriented generating sets and showed that at least four moves are necessary.
In \cite{Polyak2010}, Polyak  
showed that at least four moves are necessary.

The generating set problem of Reidemeister moves asks: 

\begin{prob}\label{prob:mgs}
What is the smallest possible subset of oriented Reidemeister moves that suffices to generate all oriented Reidemeister moves?
\end{prob}

It is known that any such generating set must contain at least one move of \tIII, one of \tII, and at least two of \tI, making four moves the absolute minimum \cite{Ostlund2001, Manturov2004, Hagge2006, Polyak2010}.  Hence, the essential problem is as follows: 

\begin{prob}
Give every minimal generating set containing a given \tIII.  
\end{prob}

\noindent{\textbf{Main Result.}}
We determine whether each known generating set of arbitrary oriented Reidemeister moves is  minimal. We then provide a complete classification of minimal generating sets that include a non-braid-type Reidemeister move of $\tII$. 
We also classify all minimal generating sets that include a braid-type Reidemeister move of $\tIII$. 
%
%Beyond these two cases, we identify $16$ possible candidates for minimal generating sets.  %Among them, we prove that $12$ are indeed minimal, whereas the minimality and even the generating property of the remaining $4$ sets remains unsolved.
%We classify all minimal  generating sets with a braid-type move of \tIII.  This resolves Problem~\ref{prob:mgs} for this class.  For non-braid-type moves of \tIII, we determine whether a generating set of size five exists in all but two case.  

Our concrete statements are presented in Section~\ref{sec:MainR}.  

See  Tables~\ref{table:classificationAH}, \ref{table:classificationBCD}, and\ref{table:classificationDEF}
for all  $336$ ($= 8 \times {}_4 C_2 \times {}_4 C_2$ $+  2 \times {}_4 C_1 \times {}_4 C_2$) cases and their status. 

%\clearpage

\begin{table}[h!]
\caption{Minimal Generating sets of Reidemeister moves.  The symbols $ \oneA, \oneB, \dots, \threeH$ of  Reidemeister moves in the table refer to move types defined in the next section.
The symbol  MGS stands for a minimal generating set.
The symbol “?” indicates that it is not known whether the set is a generating set or not.}
\label{table:classification}
\begin{center}
\begin{tabular}{|c|c|c|c|} 
    \hline
        For $ \threeA$ (non-braid-type)       &          & For $\threeH$ (non-braid-type)             &              \\
    \hline
   $\{\threeA, \twoA, \oneA, \oneB \}$ & MGS    & $\{\threeH, \twoA, \oneA, \oneB \}$ &  ?          \\
   $\{\threeA, \twoA, \oneA, \oneC \}$ & MGS    & $\{\threeH, \twoA, \oneA, \oneC \}$ & MGS         \\
   $\{\threeA, \twoA, \oneB, \oneD \}$ & MGS    & $\{\threeH, \twoA, \oneB, \oneD \}$ & MGS         \\
   $\{\threeA, \twoA, \oneC, \oneD \}$ &  ?     & $\{\threeH, \twoA, \oneC, \oneD \}$ & MGS         \\
   $\{\threeA, \twoB, \oneA, \oneB \}$ & MGS    & $\{\threeH, \twoB, \oneA, \oneB \}$ &  ?          \\
   $\{\threeA, \twoB, \oneA, \oneC \}$ & MGS    & $\{\threeH, \twoB, \oneA, \oneC \}$ & MGS         \\
   $\{\threeA, \twoB, \oneB, \oneD \}$ & MGS    & $\{\threeH, \twoB, \oneB, \oneD \}$ & MGS         \\
   $\{\threeA, \twoB, \oneC, \oneD \}$ &  ?     & $\{\threeH, \twoB, \oneC, \oneD \}$ & MGS         \\

   $\{\threeA, \twoC, \twoD, \oneA, \oneB \}$ & MGS      & $\{\threeH, \twoC, \twoD, \oneA, \oneB \}$ & MGS      \\
   $\{\threeA, \twoC, \twoD, \oneA, \oneC \}$ & MGS      & $\{\threeH, \twoC, \twoD, \oneA, \oneC \}$ & MGS      \\
   $\{\threeA, \twoC, \twoD, \oneB, \oneD \}$ & MGS      & $\{\threeH, \twoC, \twoD, \oneB, \oneD \}$ & MGS      \\
   $\{\threeA, \twoC, \twoD, \oneC, \oneD \}$ & MGS      & $\{\threeH, \twoC, \twoD, \oneC, \oneD \}$ & MGS      \\
   \hline
        For $\threeB$ (braid-type)            &          & For $\threeG$ (braid-type)                 &         
       \\
    \hline
   $\{\threeB, \twoC, \twoD, \oneA, \oneB \}$ & MGS      & $\{\threeG, \twoC, \twoD, \oneA, \oneB \}$ & MGS           \\
   $\{\threeB, \twoC, \twoD, \oneA, \oneC \}$ & MGS      & $\{\threeG, \twoC, \twoD, \oneA, \oneC \}$ & MGS           \\
   $\{\threeB, \twoC, \twoD, \oneB, \oneD \}$ & MGS      & $\{\threeG, \twoC, \twoD, \oneB, \oneD \}$ & MGS           \\
   $\{\threeB, \twoC, \twoD, \oneC, \oneD \}$ & MGS      & $\{\threeG, \twoC, \twoD, \oneC, \oneD \}$ & MGS           \\
    \hline
    For $    \threeC$ (braid-type)            &          & For   $\threeE$ (braid-type)           &              \\
    \hline
   $\{\threeC, \twoC, \twoD, \oneA, \oneB \}$ & MGS      & $\{\threeE, \twoC, \twoD, \oneA, \oneB \}$ & MGS           \\
   $\{\threeC, \twoC, \twoD, \oneA, \oneC \}$ & MGS      & $\{\threeE, \twoC, \twoD, \oneA, \oneC \}$ & MGS           \\
   $\{\threeC, \twoC, \twoD, \oneB, \oneD \}$ & MGS      & $\{\threeE, \twoC, \twoD, \oneB, \oneD \}$ & MGS           \\
   $\{\threeC, \twoC, \twoD, \oneC, \oneD \}$ & MGS      & $\{\threeE, \twoC, \twoD, \oneC, \oneD \}$ & MGS           \\
\hline
    For $     \threeD$ (braid-type)           &          & For $\threeF$ (braid-type)                 &              \\
    \hline
   $\{\threeD, \twoC, \twoD, \oneA, \oneB \}$ & MGS      & $\{\threeF, \twoC, \twoD, \oneA, \oneB \}$ & MGS           \\
   $\{\threeD, \twoC, \twoD, \oneA, \oneC \}$ & MGS      & $\{\threeF, \twoC, \twoD, \oneA, \oneC \}$ & MGS           \\
   $\{\threeD, \twoC, \twoD, \oneB, \oneD \}$ & MGS      & $\{\threeF, \twoC, \twoD, \oneB, \oneD \}$ & MGS           \\
   $\{\threeD, \twoC, \twoD, \oneC, \oneD \}$ & MGS      & $\{\threeF, \twoC, \twoD, \oneC, \oneD \}$ & MGS           \\
\hline
\end{tabular}
\end{center}
\end{table}

%\begin{rem}\label{Result-of-Polyak-Caprau-Scott}
 %   For generating sets, Polyak {\cite{Polyak2010}} proved that $\{ \threeA, \twoA, \oneA, \oneB \}$ is a minimal generating set. 
   % Subsequently, Caprau and Scott {\cite{CaprauScott2022}} proved that the four sets that are marked “MGS” in the table \ref{table:classification}, each of which consists of exactly $4$ elements, are not generating sets.If their assertions are correct, the collection of minimal generating sets is completely determined.
    %Taken together with their assertions, Table 1 constitutes a complete list.
%\end{rem}

\noindent{\textbf{Open Problem.}}
The generating set status for $\{\threeA, \twoA, \twoB, \oneC, \oneD\}$ and its mirror remains open (cf.~Remark~\ref{remhomotopy}).  

\clearpage

\section{Definitions and notations}\label{def}

\subsection{Definitions}

\begin{defn}\label{def:symbolRM}
    Oriented Reidemeister moves are defined as follows.\\
    %ここに絵
    \vspace{-2.5cm}
    \begin{figure}[h]
        \centering
        \includegraphics[width=0.95\linewidth]{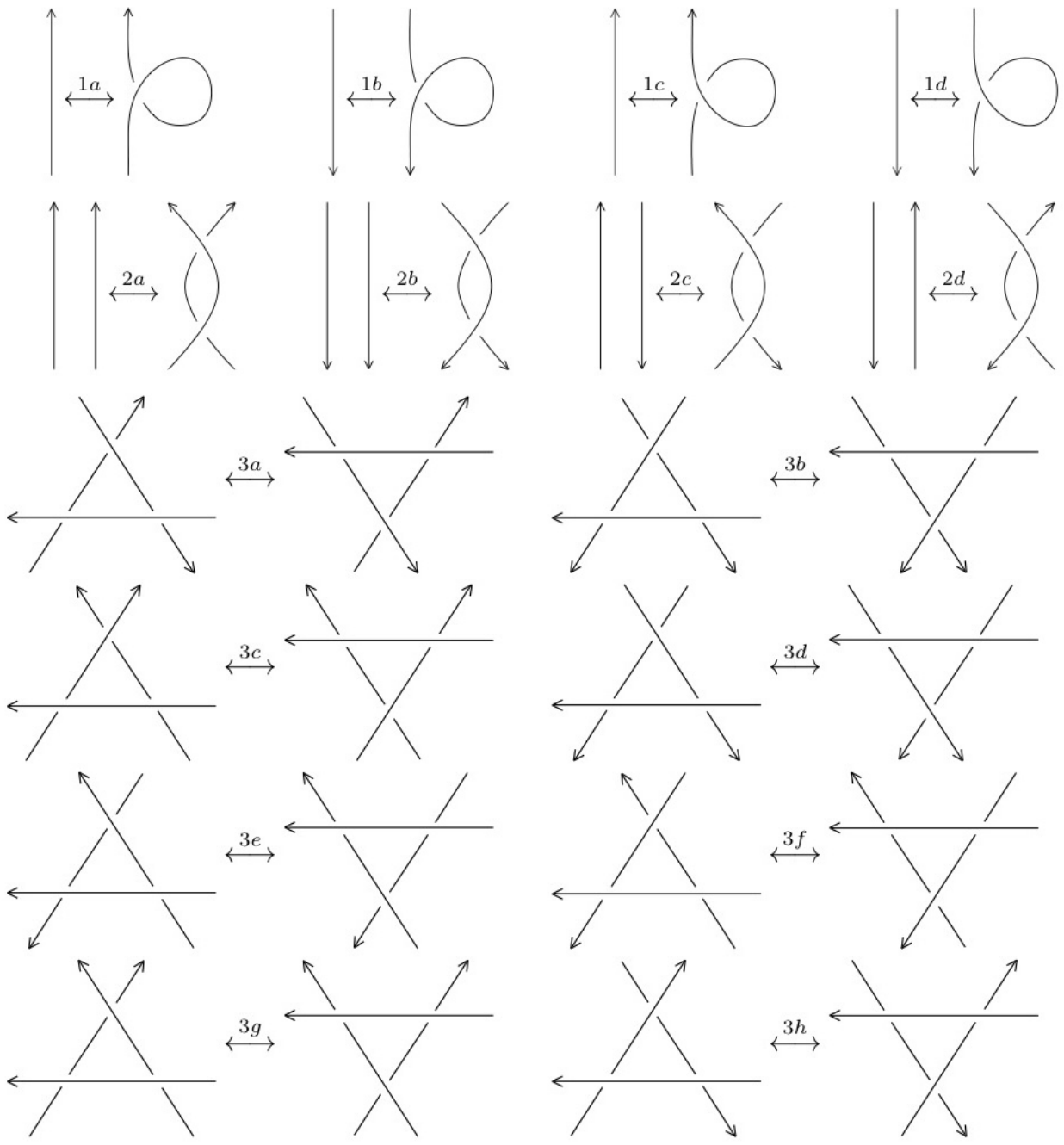}
    \end{figure}
    \vspace{-2.5cm}
    %\begin{align*}
    %  \Iu \overset{1a}{\longleftrightarrow} \Ia \quad \quad \quad \quad
    %  \Is \overset{1b}{\longleftrightarrow} \Ib \quad \quad \quad \quad
    %  \Iu \overset{1c}{\longleftrightarrow} \Ic \quad \quad \quad \quad
    %  \Is \overset{1d}{\longleftrightarrow} \Id
    %\end{align*}
    %\begin{align*}
    %  \IIuu \overset{2a}{\longleftrightarrow} \IIa \quad \quad \quad
    %  \IIss \overset{2b}{\longleftrightarrow} \IIb \quad \quad \quad
    %  \IIus \overset{2c}{\longleftrightarrow} \IIc \quad \quad \quad
    %  \IIsu \overset{2d}{\longleftrightarrow} \IId
    %\end{align*}
    %\begin{align*}
    %  \IIIao \overset{3a}{\longleftrightarrow} \IIIau \quad \quad
    %  \IIIbo \overset{3b}{\longleftrightarrow} \IIIbu \quad \quad \\
    %  \IIIco \overset{3c}{\longleftrightarrow} \IIIcu \quad \quad
    %  \IIIdo \overset{3d}{\longleftrightarrow} \IIIdu \quad \quad \\
    %  \IIIeo \overset{3e}{\longleftrightarrow} \IIIeu \quad \quad
    %  \IIIfo \overset{3f}{\longleftrightarrow} \IIIfu \quad \quad \\
    %  \IIIgo \overset{3g}{\longleftrightarrow} \IIIgu \quad \quad
    %  \IIIho \overset{3h}{\longleftrightarrow} \IIIhu \quad \quad
    %\end{align*}
\end{defn}
    The moves of \tIII~are classified into braid-type and non-braid-type: $\threeA$ and $\threeH$ are non-braid-type, while $\threeB, \threeC, \threeD, \threeE, \threeF, \threeG$ are braid-type.
    Similarly, among the moves of \tII, $\twoC$ and $\twoD$ are non-braid-type, and $\twoA, \twoB$ are braid-type.
\begin{defn}
    We denote by $\R I, \R \II, \R \III$ the set of all move of \tI, \tII, \tIII,
    and let $\R \coloneq \R I \cup \R \II \cup  \R \III$. 
\end{defn}    
\begin{defn}
    We call a set $S \subset \R$  a \emph{generating set}, 
    if any  move in $\R$ is obtained by a finite sequence of isotopies and from the set $S$ inside the changing disk of the move.
\end{defn}
    We denote by $GS$ the set of all generating sets. 
    For $\ast \in \{ \textup{a, b, c, d, e, f, g, h} \}$, 
    we define $GS_\ast \coloneq \{ S \in GS \mid S \cap \R \III  = \{ \threeAS \}  \}$. 
    Note that $GS_a \cup GS_b \cup \dots \cup GS_h \subsetneq GS$.
\begin{defn}\label{def:minimal}
    The set $S \in GS$ is a \emph{minimal generating set} 
    if for any $m \in S$, $ S \setminus \{m\} \notin GS .$ 
    For $\ast \in \{a, b, c, d, e, f, g, h \}$, 
    we denote $\textup{MGS}_\ast \coloneq \{ S \in GS_\ast \mid  S:minimal \}$.
    Let $\textup{MGS} \coloneq \textup{MGS}_a \cup \textup{MGS}_b \cup \dots \cup \textup{MGS}_h$.
\end{defn}

\begin{defn}
    We say that a set $S\in GS$ is coherent $\textup{\II}$ if, among $\tII$ moves,
    S contains precisely $\twoC$ and $\twoD$ $($i.e., neither $\twoA$ nor $\twoB$$)$.
    Likewise, $S$ is coherent $\textup{\III}$ if, among the non-braid $\tIII$ moves, 
    S contains exactly one $($i.e., $\threeA$ or $\threeH$$)$.
\end{defn}

\begin{defn}
    At every crossing in the before-and-after diagrams, switch the over- and under-strands $($i.e., perform a crossing switch at every crossing$)$.
    We call this ``\emph{Mirror}''.
\end{defn}

\begin{center}
\begin{tabular}{|c|c|c|} \hline
   Mirror pair of $\tI$ & Mirror pair of $\tII$ &   Mirror pair of $\tIII$ \\ \hline
    $\oneA   \overset{\text{Mirror}}{\longleftrightarrow} \oneC$ 
   &$\twoA   \overset{\text{Mirror}}{\longleftrightarrow} \twoB$
   &$\threeA \overset{\text{Mirror}}{\longleftrightarrow} \threeH$ \\
    $\oneB   \overset{\text{Mirror}}{\longleftrightarrow} \oneD$ 
   &$\twoC   \overset{\text{Mirror}}{\longleftrightarrow} \twoC$
   &$\threeB \overset{\text{Mirror}}{\longleftrightarrow} \threeG$ \\
    
   &$\twoD   \overset{\text{Mirror}}{\longleftrightarrow} \twoD$
   &$\threeC \overset{\text{Mirror}}{\longleftrightarrow} \threeE$ \\
    
   &
   & $\threeD \overset{\text{Mirror}}{\longleftrightarrow} \threeF$ \\   
    \hline
\end{tabular}
\end{center}

\begin{cor}
    If $S$ is a generating set, 
    the set $S^{\prime}$, consisting of the Mirror of the elements of S, is also a generating set.  
\end{cor}

\subsection{Facts}
\begin{defn}
    We call the pairs $(\oneA, \oneB),(\oneA, \oneC), (\oneB, \oneD)$ and $(\oneC, \oneD)$ compatible pairs.
\end{defn}

\begin{lem}[Polyak,  {\cite[Lemma~3.1]{Polyak2010}}]\label{[Pol]3.1}
       Any generating set must contain at least one compatible pair.
\end{lem}

It follows from lemma~\ref{[Pol]3.1} that any generating set contains at least four moves: two of \tI, one of \tII, and one of \tIII.

\begin{thm}[Polyak,  {\cite[Theorem~1.1]{Polyak2010}}]\label{Pol-thm1}
    The set $\{ \threeA,\twoA, \oneA, \oneB \}$ is a minimal generating set.

    %\vspace{-5cm}
    \hspace{-1.5cm}
    \includegraphics[width=1.1\linewidth]{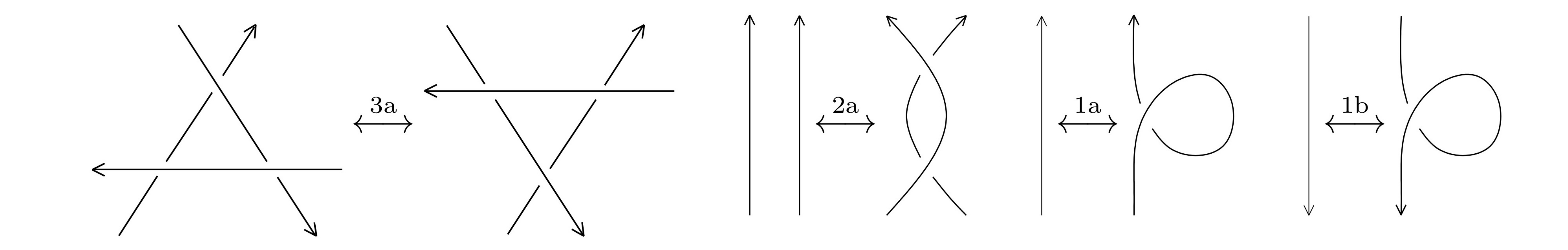}

    %\vspace{-4.5cm}
    %\begin{align*}
    %  \IIIao
    %  \overset{\threeA}{\longleftrightarrow} \IIIau \quad \quad 
    %  \IIuu 
    %  \overset{\twoA}{\longleftrightarrow}   \IIa   \quad \quad 
    %  \Iu   
    %  \overset{\oneA}{\longleftrightarrow}   \Ia    \quad \quad 
    %  \Is   
    %  \overset{\oneB}{\longleftrightarrow}   \Ib   
    %\end{align*}
\end{thm}

The set considered in Theorem~\ref{Pol-thm1} is coherent $\textup{\III}$ but not coherent $\text{\II}$.

\begin{thm}[Polyak,  {\cite[Theorem~1.2]{Polyak2010}}]\label{thm:PolyakBraid}
    We only consider sets $S$ whose $\tI$ moves constitute a compatible pair. 
    A set $S$ containing only one move,  $\threeB$, of \tIII $\!$ is a generating set if and only if $S$ includes $\twoC$ and $\twoD$, and contains one of the compatible pairs.
\end{thm}

The set considered in Theorem~\ref{thm:PolyakBraid} is coherent $\text{\II}$ but not coherent $\text{\III}$.

We prove a theorem that extends the assertion of Theorem~\ref{thm:PolyakBraid} for the move 3b to the other braid-type {\textup{\III}} moves; see Theorem~\ref{Extend-of-Polyak}.
%\begin{defn}[fundamental groupoid]
%For a topological space $X$, we define the groupoid  $\Pi_1 (X)$ by 
%\begin{itemize}
%\item objects are elements of $X$ 
%\item morphisms from $x$ to $y$ are equivalence classes of homotopy of homotopy relative to $\partial I[\gamma]$ of continuous maps $\gamma : [0, 1] \to X$ whose endpoints map to $x$ and $y$ which homotopies are required to fix.  The composition is by concatonation of representative maps.  
%This becomes an associative and unital composition with respect to which every morphism has an inverse.   
%\end{itemize}
%\emph{fundamental groupoid}
%\end{defn}
\section{Main results}\label{sec:MainR}
%%
%If the set $S$ containing three 1-moves, the $S$ is not the minimal generating set.  
%%
Let $CP$ be any compatible pair, i.e, $(\oneA, \oneB),(\oneA, \oneC), (\oneB, \oneD)$ or $(\oneC, \oneD)$. 
Let $\R$ be the set of the all oriented Reidemeister moves, $\R I$ the set of the all moves of \tI,  $\R \II$ the set of the all moves of \tII , and $\R \III$ the set of the all moves of \tIII.  

\begin{prop}\label{RI-gene}
    The set 
    $ \{ \twoC, \twoD \} \cup CP $ 
    generates all elements of $\R I$.
\end{prop}

\begin{prop}\label{3ast-two2}
    For any   $ \ast \in \{ \textup{c, d, e, f} \} $,
    the set   
    $ \{ \threeAS \} \cup \{ \twoC, \twoD \} \cup CP$
    generates  all elements of $\R \II$.
    For any   $ \ast \in \{ \textup{a, h} \} $,
    the set   
    $ \{ \threeAS \} \cup \{ \twoA, \twoB \} \cup CP$ 
    generates  all elements of $\R \II$.
\end{prop}

\begin{prop}\label{3ast-all2}
    For every $\threeAS$, 
    the set   $ \{ \threeAS \} \cup \{ \twoA,\twoB, \twoC, \twoD \} $ 
    generates all elements of $\R \III$.
\end{prop}

    Propositions~\ref{RI-gene},\ref{3ast-two2},\ref{3ast-all2} imply  Corollary~\ref{Cor:3*Gen}. 
\begin{cor}\label{Cor:3*Gen}
The following two statements hold.  
    \begin{itemize}
        \item  
        For any $ \ast \in \{ \textup{c, d, e, f} \} $, 
        the set
        $ \{ \threeAS \} \cup \{ \twoC, \twoD \} \cup CP$
        is a  generating set.
        \item   
        For any $ \ast \in \{ \textup{a, h} \} $,
        the set   
        $ \{ \threeAS \} \cup \{ \twoA, \twoB \} \cup CP$ 
        is also a generating set.
    \end{itemize}
\end{cor}

\begin{thm}\label{prop:cdAlone}
    Any minimal generating set containing  $\twoC$ must also contain $\twoD$, and vice versa.  
    If a generating set contains $\twoC$ but neither $\twoA$ nor $\twoB$, then it must also contain $\twoD$\, and vice versa.
\end{thm}
\begin{thm}\label{Extend-of-Polyak}
For any set $S \subset \R$ whose $\tI$ moves constitute a compatible pair, the following two statements hold. 
\begin{itemize}
    \item Let $\ast \in \{ b, c, d, e, f, g\}$.
        \begin{itemize}
            \item[(1)] A set $S$ containing $\threeAS$ is a \emph{minimal} generating set if and only if $S = \{\threeAS, \twoC, \twoD\} \cup CP$.
            \item[(1)$^{\prime}$]  A set $S$ containing $\threeAS$ is a generating set if and only if $S \supset \{\threeAS, \twoC, \twoD\} \cup CP$.
        \end{itemize}
    \item 
    Let $\ast \in  \{ a, h \}$.
    \begin{itemize}
        \item[(2)]
      A set $S$ containing  $\threeAS$ is a \emph{minimal} generating set if  $S = \{\threeAS, \twoC, \twoD \} \cup CP$.  
    \item[(2)$^{\prime}$] A set $S$ containing  $\threeAS$ is a  generating set if  $S \supset \{\threeAS, \twoC, \twoD \} \cup CP$.
    \end{itemize}
\end{itemize}
\end{thm}

By explicitly listing generating sets that satisfy (1) and (2) of Theorem~\ref{Extend-of-Polyak}, we obtain Corollary~\ref{cor:List}.
%More concretely, 
\begin{cor}\label{cor:List}
Let $\textup{MGS}_{\ast}$ be as in Definition~\ref{def:minimal}.  The list of five-element minimal generating sets is as follows.  
\begin{align*}
&\textup{MGS}_a 
\supset 
\{
    \{\threeA, \twoC, \twoD, \oneA, \oneB \}, \{\threeA, \twoC, \twoD, \oneA, \oneC \}, \{\threeA, \twoC, \twoD, \oneB, \oneD \}, \{\threeA, \twoC, \twoD, \oneC, \oneD \} 
\}, \\
&\textup{MGS}_b 
= 
\{
    \{\threeB, \twoC, \twoD, \oneA, \oneB \}, \{\threeB, \twoC, \twoD, \oneA, \oneC \}, \{\threeB, \twoC, \twoD, \oneB, \oneD \}, \{\threeB, \twoC, \twoD, \oneC, \oneD \} 
\}, \\
&\textup{MGS}_c 
= 
\{
\{\threeC, \twoC, \twoD, \oneA, \oneB \}, \{\threeC, \twoC, \twoD, \oneA, \oneC \}, \{\threeC, \twoC, \twoD, \oneB, \oneD \}, \{\threeC, \twoC, \twoD, \oneC, \oneD \}
\}, \\
&\textup{MGS}_d 
= 
\{
\{\threeD, \twoC, \twoD, \oneA, \oneB \}, \{\threeD, \twoC, \twoD, \oneA, \oneC \}, \{\threeD, \twoC, \twoD, \oneB, \oneD \}, \{\threeD, \twoC, \twoD, \oneC, \oneD \}
\}, \\
&\textup{MGS}_e 
= 
\{ 
\{\threeE, \twoC, \twoD, \oneA, \oneB \}, \{\threeE, \twoC, \twoD, \oneA, \oneC \}, \{\threeE, \twoC, \twoD, \oneB, \oneD \}, \{\threeE, \twoC, \twoD, \oneC, \oneD \}
\}, \\
&\textup{MGS}_f 
= 
\{
\{\threeF, \twoC, \twoD, \oneA, \oneB \}, \{\threeF, \twoC, \twoD, \oneA, \oneC \}, \{\threeF, \twoC, \twoD, \oneB, \oneD \}, \{\threeF, \twoC, \twoD, \oneC, \oneD \} 
\}, \\
&\textup{MGS}_g 
= 
\{
\{\threeG, \twoC, \twoD, \oneA, \oneB \}, \{\threeG, \twoC, \twoD, \oneA, \oneC \}, \{\threeG, \twoC, \twoD, \oneB, \oneD \}, \{\threeG, \twoC, \twoD, \oneC, \oneD \}
\},\\
&\textup{MGS}_h 
\supset 
\{
    \{\threeH, \twoC, \twoD, \oneA, \oneB \}, \{\threeH, \twoC, \twoD, \oneA, \oneC \}, \{\threeH, \twoC, \twoD, \oneB, \oneD \}, \{\threeH, \twoC, \twoD, \oneC, \oneD \} 
\}. \\
\end{align*}
\end{cor}

\begin{rem}The following should be noted.
    \begin{itemize}
        \item $(1)$ of Theorem~\ref{Extend-of-Polyak} shows that a set S containing one braid $\tIII$ is a minimal generating set if and only if it is coherent $\textup{\II}$.
        \item $(2)$ of Theorem~\ref{Extend-of-Polyak} shows that a set S containing one non-braid $\tIII$ is a minimal generating set if  it is coherent $\textup{\II}$.
        However, the equivalence fails: there exist minimal generating sets containing $\threeA$ or $\threeH$, each of which is not coherent $\textup{\II}$.
    \end{itemize}
\end{rem}

\begin{thm}\label{Thm:5elements}
    There is no minimal generating set with six or more elements.
    A set $S \subset \R$ is a  five-element minimal generating set 
    if and only if $S \in \textup{MGS}_{\ast}$ with $\ast \in \{b, c, d, e, f, g\}$, or $S \in \textup{MGS}_{\ast}$ with $\ast \in \{a, h \}$ and is coherent~\II. 
    %Moreover, any minimal generating set that contains a braid $\tIII$ move has cardinality five.
\end{thm}

\section{Proofs}\label{proof}
\subsection{Lemmas on generating oriented Reidemeister moves}
    In this section, we summarize, by type, which Reidemeister moves can be generated from which sets of moves.
    We write 
    $ m \prec S $
    to indicate that the move $m \in \R $ is generated by the elements of the set $ S \subset \R $.
    %For readability in this section, and when there is no risk of confusion, we permit the following notation:  we allow the notation 
    %\[ x, y, z,\dots \prec S \] instead of $\{ x, y, z, \dots \} \prec S$.   
\subsubsection{Generating \tI}

\begin{lem}\label{Generating1}
    Moves of $\tI$ can be generated as follows.
  \begin{align}
    \oneA \prec \{ \oneD, \twoD \}.\\
    \oneB \prec \{ \oneC, \twoC \}.\\
    \oneC \prec \{ \oneB, \twoD \}.\\
    \oneD \prec \{ \oneA, \twoC \}.
  \end{align}
\end{lem}

\begin{proof}
For the first case, $\oneA$ can be generated by the following sequence.
The same applies to the other $\tI$ moves as in Figure\ref{movie1}.
  \begin{figure}[h]
      \center
      \includegraphics[width=0.75\linewidth]{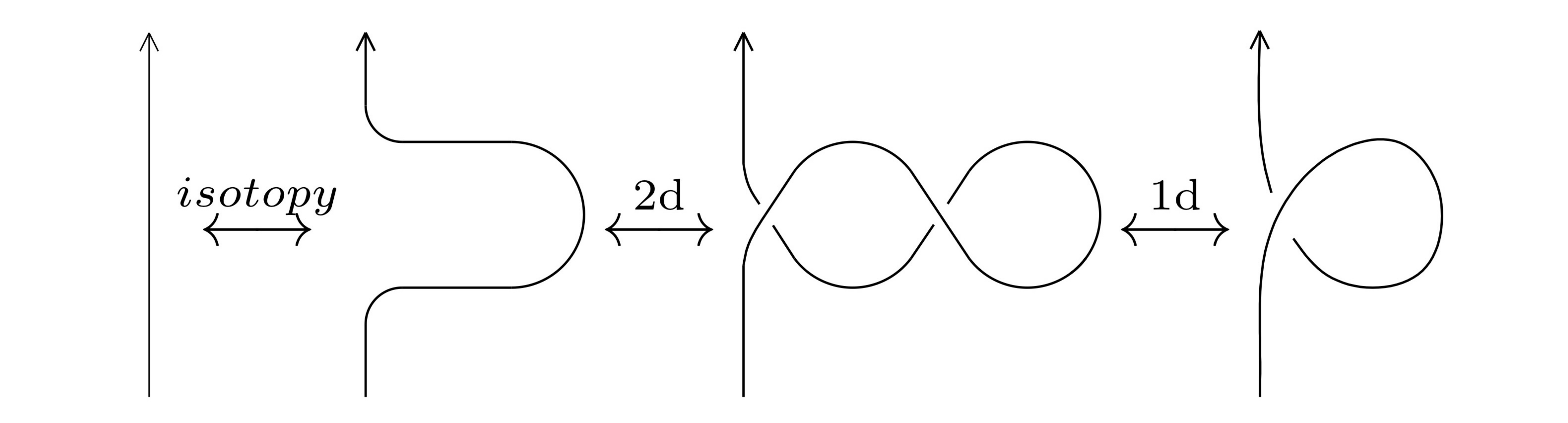}
      \caption{}
      \label{movie1}
  \end{figure}
\end{proof}
    %\vspace{-3cm}

    %\begin{align*}
    %    \Iu 
    %    \overset{isotopy}{\longleftrightarrow}
    %    \Iuiso
    %    \overset{\twoD}{\longleftrightarrow}
    %    \IuIId
    %    \overset{\oneD}{\longleftrightarrow}
    %    \Ia
    %\end{align*}

%\vspace{-3cm}

\begin{rem}
The following diagram is obtained from Lemma~\ref{Generating1}.
The pairs ($\oneA, \oneD$) and ($\oneB, \oneC$) can be transformed into each other by applying either 2c or 2d.
\end{rem}

\vspace{-2.5cm}
\begin{center}
    \includegraphics[width=0.6\linewidth]{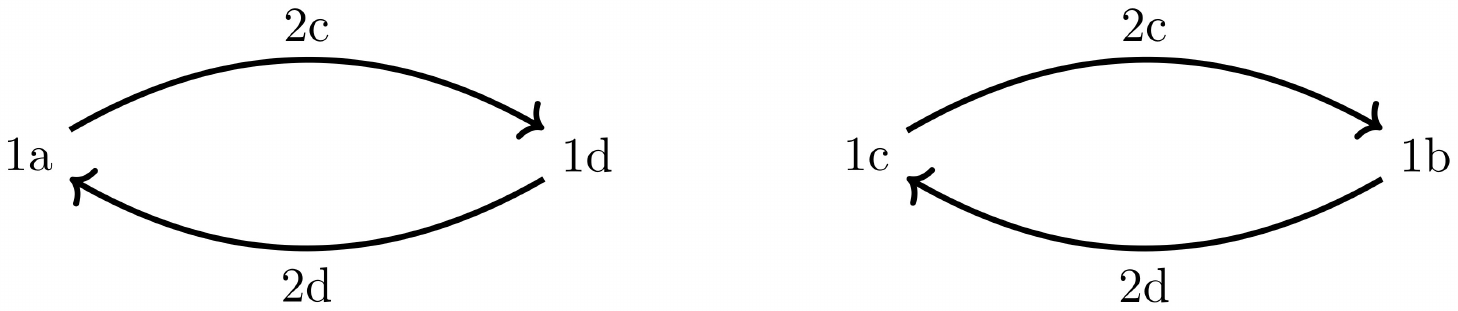}
\end{center}  
\vspace{-2cm}

%\begin{center}
%\begin{tikzpicture}[->]
%  \node (a) at (0, 0)  {$\oneA$}; 
%  \node (d) at (4, 0)  {$\oneD$};
%  \node (c) at (6, 0)  {$\oneC$}; 
%  \node (b) at (10, 0)  {$\oneB$};
%  \draw[very thick] (a) to[bend left]node[above]{$\twoC$} (d) ;
%  \draw[very thick] (d) to[bend left]node[below]{$\twoD$} (a) ;
%  \draw[very thick] (c) to[bend left]node[above]{$\twoC$} (b) ;
%  \draw[very thick] (b) to[bend left]node[below]{$\twoD$} (c) ;
%\end{tikzpicture}
%\end{center}

Proposition~\ref{RI-gene} is proved by Lemma~\ref{Generating1}.

\subsubsection{Generating \tII}

\begin{lem}\label{Generating2}
    Moves of $\tII$ can be generated as follows.
    \begin{align}
      \twoA \prec \{ \oneA, \twoD, \threeD \},
      \twoA \prec \{ \oneB, \twoC, \threeE \},
      \twoA \prec \{ \oneC, \twoD, \threeC \},
      \twoA \prec \{ \oneD, \twoC, \threeF \}.\\
      \twoB \prec \{ \oneA, \twoD, \threeE \}, 
      \twoB \prec \{ \oneB, \twoC, \threeD \}, 
      \twoB \prec \{ \oneC, \twoD, \threeF \}, 
      \twoB \prec \{ \oneD, \twoC, \threeC \}.\\ 
      \twoC \prec \{ \oneA, \twoA, \threeA \},
      \twoC \prec \{ \oneA, \twoB, \threeA \},
      \twoC \prec \{ \oneC, \twoA, \threeH \},
      \twoC \prec \{ \oneC, \twoB, \threeH \}.\\
      \twoD \prec \{ \oneB, \twoA, \threeA \},
      \twoD \prec \{ \oneB, \twoB, \threeA \},
      \twoD \prec \{ \oneD, \twoA, \threeH \},
      \twoD \prec \{ \oneD, \twoB, \threeH \}.
    \end{align}
\end{lem}

\begin{proof}
For the first case, $\twoA$ can be generated by the following sequence.
The same applies to the other $\tII$ moves as in Figure~\ref{movie2}.
    %\begin{align*}
    %    \IIuu 
    %    \overset{\oneA}{\longleftrightarrow}
    %    \Ia \quad \Iu
    %    \overset{\twoD}{\longleftrightarrow}
    %    \IIuuIaIId
    %    \overset{\threeD}{\longleftrightarrow}
    %    \IIuuIaIIdIIId
    %    \overset{\oneA}{\longleftrightarrow}
    %    \IIa
    %\end{align*}
   \begin{figure}[h]
      \center
      \includegraphics[width=0.9\linewidth]{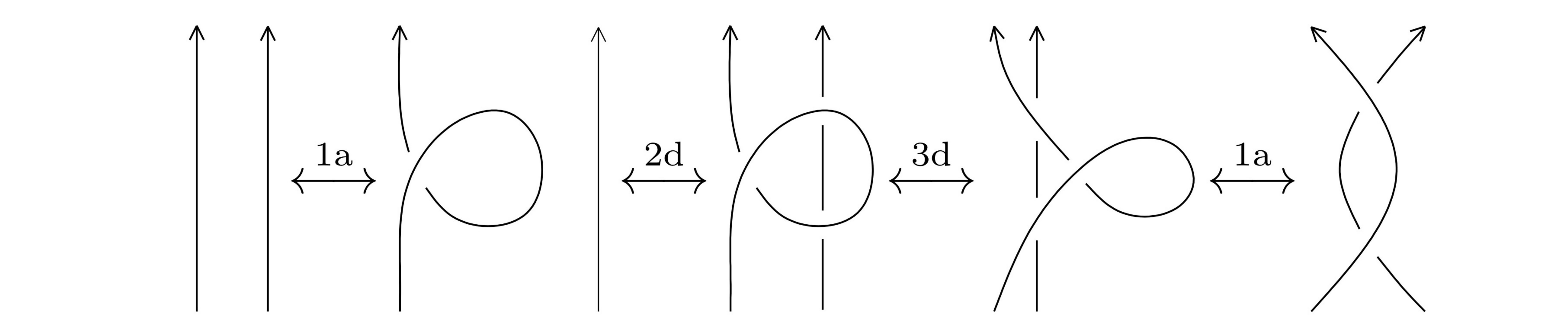}
      \caption{}
      \label{movie2}
  \end{figure}
\end{proof}

\begin{rem}
The following diagram is obtained from Lemma~\ref{Generating2}.
The sequence of moves in the proof of Lemma~\ref{Generating2} does not allow a transformation between $\twoA$ and $\twoB$, or between $\twoC$ and $\twoD$.
\end{rem}

%\begin{center}
%\begin{tikzpicture}[->]
%  \node (a) at (0, 6) {$2a$}; 
%  \node (d) at (6, 6) {$2d$};
%  \node (c) at (0, 0) {$2c$}; 
%  \node (b) at (6, 0) {$2b$};
% \draw[very thick] (a) to[bend left = 10] node[sloped, above]{\{ \threeA, \oneA \}, \{ \threeH, \oneC \} } %%(c);
%  \draw[very thick] (a) to[bend left = 10] node[sloped, above]{\{ \threeA, \oneB \}, \{ \threeH, \oneD \} } (d);
%  \draw[very thick] (b) to[bend left = 10] node[sloped, below]{\{ \threeA, \oneA \}, \{ \threeH, \oneC \} } (c);
%  \draw[very thick] (b) to[bend left = 10] node[sloped, above]{\{ \threeA, \oneB \}, \{ \threeH, \oneD \} } (d);
%  \draw[very thick] (c) to[bend left = 10] node[sloped, above]{\{ \threeE, \oneB \}, \{ \threeF, \oneD \} } (a);
%  \draw[very thick] (c) to[bend left = 10] node[sloped, above]{\{ \threeD, \oneB \}, \{ \threeC, \oneD \} } (b);
%  \draw[very thick] (d) to[bend left = 10] node[sloped, below]{\{ \threeD, \oneA \}, \{ \threeC, \oneC \} } (a);
%  \draw[very thick] (d) to[bend left = 10] node[sloped, above]{\{ \threeE, \oneA \}, \{ \threeF, \oneC \} } (b);
%\end{tikzpicture}
%\end{center}
\begin{center}
    \includegraphics[width=0.5\linewidth]{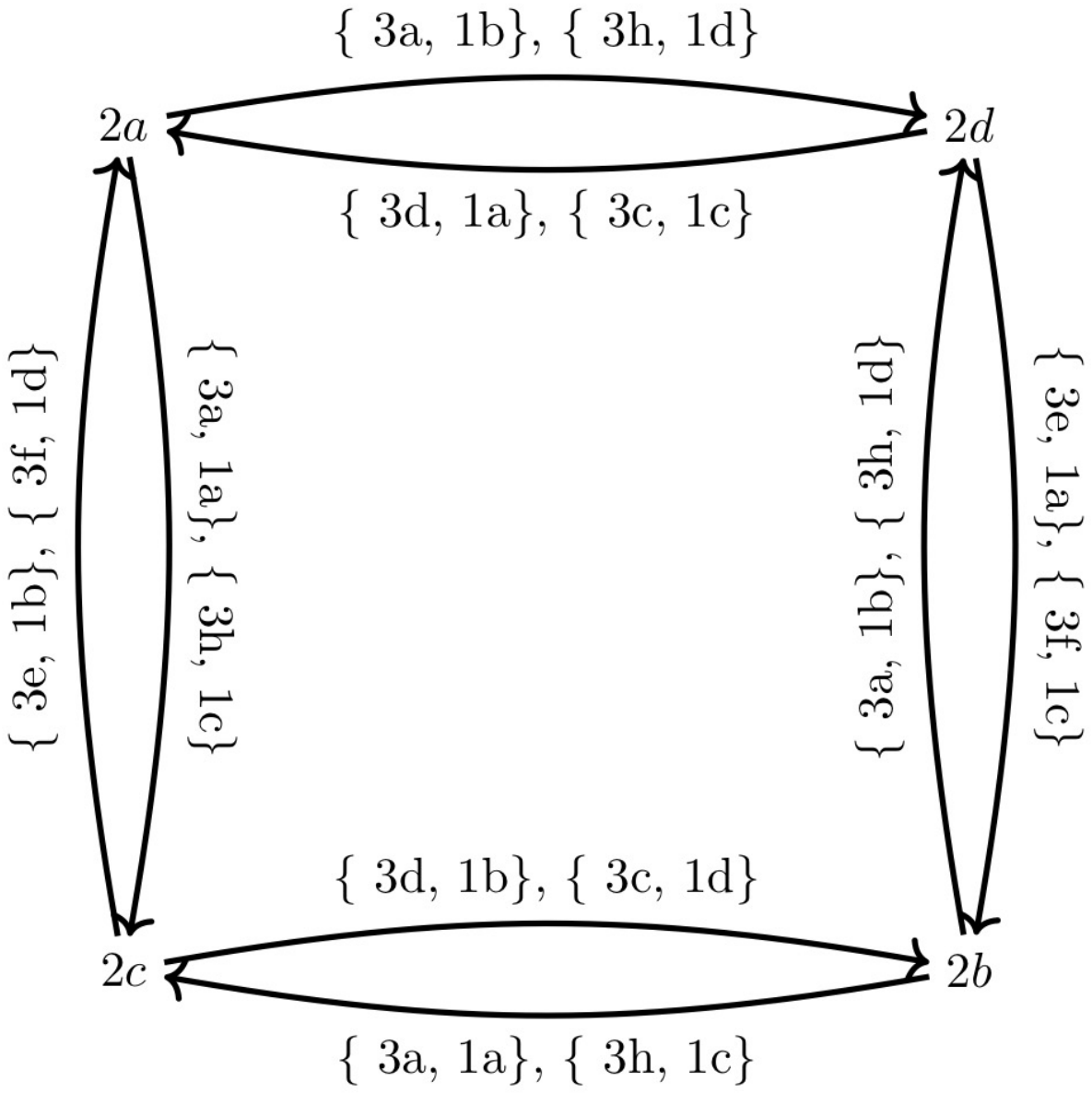}
\end{center}

Proposition~\ref{3ast-two2} is proved by Lemma~\ref{Generating2}.

\subsubsection{Generating \tIII}

\begin{lem}\label{Generating3}
Moves of $\tIII$ can be generated as follows.
  \begin{align}
    \threeA \prec \{ \twoC, \twoD, \threeC \},
    \threeA \prec \{ \twoC, \twoD, \threeB \},
    \threeA \prec \{ \twoC, \twoD, \threeF \}.\\
    \threeB \prec \{ \twoC, \twoD, \threeA \},
    \threeB \prec \{ \twoA, \twoB, \threeE \},
    \threeB \prec \{ \twoB, \twoA, \threeD \}.\\
    \threeC \prec \{ \twoA, \twoB, \threeE \},
    \threeC \prec \{ \twoC, \twoD, \threeA \}, 
    \threeC \prec \{ \twoA, \twoB, \threeG \}.\\
    \threeD \prec \{ \twoC, \twoD, \threeH \}, 
    \threeD \prec \{ \twoA, \twoB, \threeF \}, 
    \threeD \prec \{ \twoA, \twoB, \threeB \}.\\
    \threeE \prec \{ \twoA, \twoB, \threeB \}, 
    \threeE \prec \{ \twoA, \twoB, \threeC \}, 
    \threeE \prec \{ \twoC, \twoD, \threeH \}.\\
    \threeF \prec \{ \twoA, \twoB, \threeD \},
    \threeF \prec \{ \twoA, \twoB, \threeG \}, 
    \threeF \prec \{ \twoD, \twoC, \threeA \}.\\
    \threeG \prec \{ \twoA, \twoB, \threeF \},
    \threeG \prec \{ \twoC, \twoD, \threeH \},
    \threeG \prec \{ \twoB, \twoA, \threeC \}.\\
    \threeH \prec \{ \twoC, \twoD, \threeG \},
    \threeH \prec \{ \twoC, \twoD, \threeD \},
    \threeH \prec \{ \twoC, \twoD, \threeE \}.
  \end{align}
\end{lem}

\begin{proof}
For the first case, $\threeA$ can be generated by the following sequence.
The same applies to the other $\tIII$ moves as in Figure~\ref{movie3}.
    \begin{figure}[h]
      \center
      \includegraphics[width=1\linewidth]{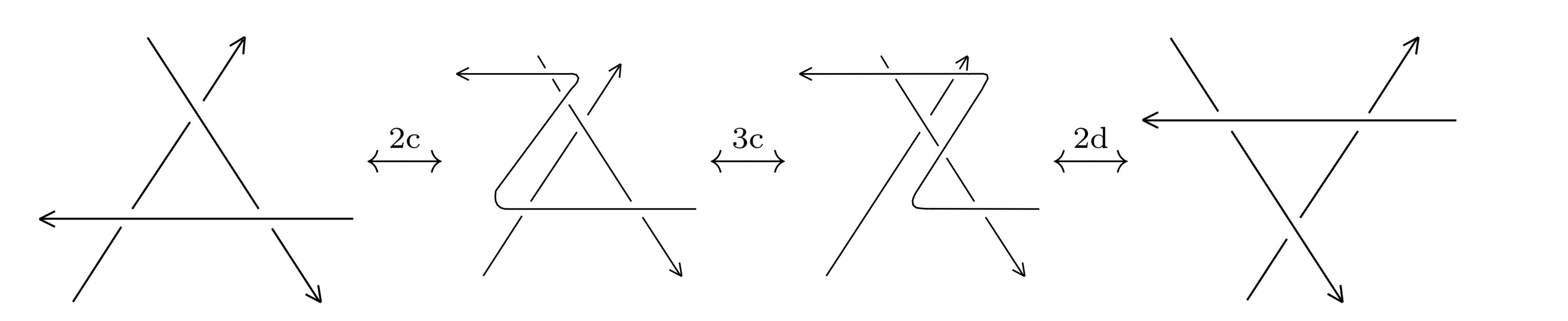}
      \caption{}
      \label{movie3}
  \end{figure}
\end{proof}

\begin{rem}
The following diagram is obtained from Lemma~\ref{Generating3}.
The blue arrows represent the transitions realized by the moves $\twoA$ and $\twoB$, whereas red arrows represent transitions realized by the moves $\twoC$ and $\twoD$.
We observe that transforming between non-braid-type moves ($\threeA, \threeH$) and braid-type moves ($\threeB, \threeC, \threeD, \threeE, \threeF, \threeG$) require the use of $\twoC$ and $\twoD$, while transformations among braid-type moves themselves can be achieved using $\twoA$ and $\twoB$.
\end{rem}

%\begin{center}
%\tikzset{cross/.style={preaction={-,draw=white,line width=6pt}}}
%\begin{tikzpicture}[<->]
%\node (c) at (0,5,0) {$\threeC$}; \node (e) at (5,5,0) {$\threeE$};
%\node (g) at (0,0,0) {$\threeG$}; \node (h) at (5,0,0) {$\threeH$};
%\node (a) at (0,5,5) {$\threeA$}; \node (b) at (5,5,5) {$\threeB$};
%\node (f) at (0,0,5) {$\threeF$}; \node (d) at (5,0,5) {$\threeD$};
%\draw[very thick, blue] (c) --(e);
%\draw[very thick, blue] (e) --(b);
%\draw[very thick, blue] (c) --(g);
%\draw[very thick, red ] (g) --(h);
%\draw[very thick, cross, red ] (a) --(b);
%\draw[very thick, cross, blue] (b) --(d);
%\draw[very thick, red ] (a) --(f);
%\draw[very thick, blue] (f) --(d);
%\draw[very thick, red ] (c) --(a);
%\draw[very thick, blue] (g) --(f);
%\draw[very thick, red ] (e) --(h);
%\draw[very thick, red ] (h) --(d);
%\end{tikzpicture}
%\end{center}

\begin{center}
    \includegraphics[width=0.5\linewidth]{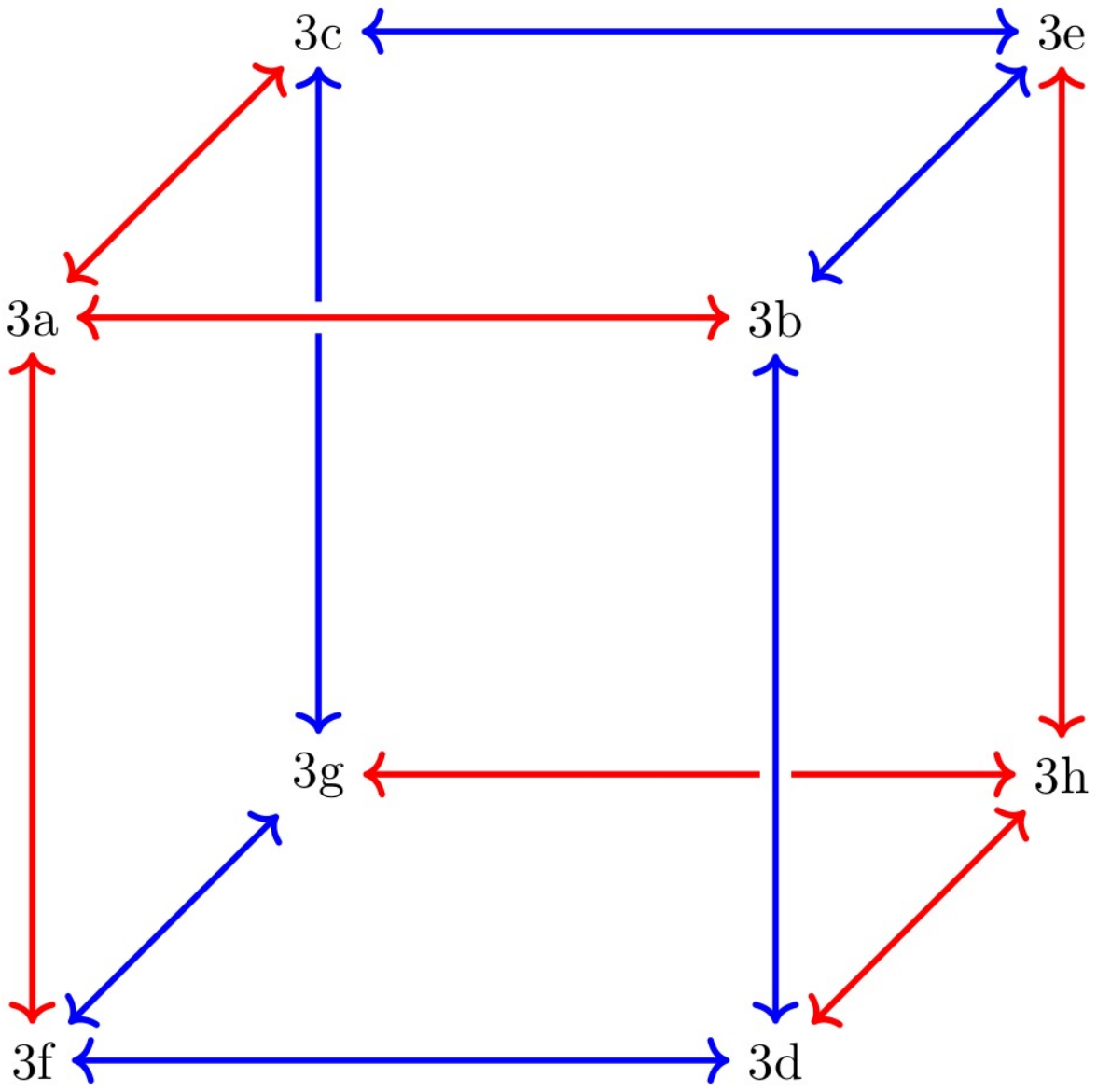}
\end{center}

Proposition~\ref{3ast-all2} is proved by Lemma~\ref{Generating3}.

\subsection{Proof of Theorem~\ref{Thm:5elements}}
In this section, we prove Theorem~\ref{Thm:5elements}. 
Since no set with six or more elements can be minimal, 
and since Theorem~\ref{Extend-of-Polyak}
(which forces the presence of 2c and 2d) 
together with Lemma~\ref{[Pol]3.1}
(necessity of type I moves) 
implies 
that any five- elements minimal generating set must consist of exactly one  \tIII~move, two \tII~moves, and two \tI~moves.

\begin{proof}
    It follows from Theorem~\ref{Extend-of-Polyak} that any minimal generating set containing a braid-\III~move has five elements.
    Hence, in the non-braid-\tIII~case, we prove below that no minimal generating set has six or more elements.
    Suppose, for a contradiction, that there exists a minimal generating set $S \in MGS_a$ that has six or more elements. By Lemma~\ref{[Pol]3.1}, $S$ contains a compatible pair $1_{\ast},1_{\sharp}$ of type I moves and suppose that 
    $\{1_{\ast},1_{	\sharp} \} \neq \{\oneC,\oneD\} $ (cf.~{\bf Open~Problem} in Section~\ref{Intro}). 
    We distinguish cases according to which $\tII$ moves lie in $S$.
    \begin{itemize}
        \item[(i)]
             If $\twoA \in S$, then $\{\threeA, \twoA, \oneAS, \oneSH \}\subset S$ is (known to be) a four-elements minimal generating set, contradicting the six-or-more-elements minimality of $S$.
        \item[(ii)] 
             If $\twoB \in S$, the same argument as in (i) applies with $2b$ in place of $2a$.
        \item[(iii)]
             If $\twoC \in S$ while $\twoA, \twoB \notin S$
             \begin{itemize}
                \item[(A)]
                     If $\twoD \in S$, then $ \{\threeA, \twoC, \twoD, \oneAS, \oneSH \} \subset S$ is a five-element minimal generating set (Theorem~\ref{Extend-of-Polyak}), again contradicting the minimality of $S$.
                \item[(B)]
                    If $\twoD \notin S$, then by Theorem~\ref{prop:cdAlone} the set $S$ is not generating, a contradiction.
             \end{itemize}
        \item[(iv)]
             If $\twoD \in S$ while $\twoA, \twoB \notin S$, the argument is symmetric to (iii).
    \end{itemize}
    These contradictions show that no minimal generating set can have six or more elements. The same proof works with $\threeH$ in place of $\threeA$. 
\end{proof}

\subsection{Non-generating oriented Reidemeister moves}\label{sec:nonGen} 
The objective of this section is to prove Theorems~\ref{prop:cdAlone} and~\ref{Extend-of-Polyak}.  
%We now define the technical function used in the proof of Theorem~\ref{prop:cdAlone}.    
\begin{defn}\label{def:monoid}
Let $\alpha$ be a finite set and the free  monoid $\pi(\alpha)$ generated by $\alpha$  consists of all finite words over $\alpha$,   with the identity  $\emptyset$, and  multiplication  by  concatenation.  
%Let $\mathbb{Z}[\pi(\alpha)_k]$ be the $\mathbb{Z}$-module of the monoid $\pi(\alpha)_k$.
\end{defn}
\begin{defn}
Let $D_L$ be a two-component link diagram with an ordering of its components and a base point on each component.   
Let $c^1_2(D_L)$ denote the totally ordered set of crossings where the over-path belongs to the first component and the under-path belongs to the second component, with the order given as $c_1, c_2, \dots, c_{\ell}$ by following the orientation of the first component starting from its base point.   
\end{defn}
\begin{defn}
Let $\alpha_* = \{ a_+, b_- \}$.  
Define $p : c^1_2(D_L) \to \alpha_*$ by  mapping  crossings
\[
\ap, \quad \bm
\]
to $a_+$, $b_-$, respectively.
\end{defn}
\begin{ex}
For the link diagram shown in Figure~\ref{fig:Hopf}, $c^1_2(D_L)$ consists of three crossings and $c^1_2 (D_L)=\{A, C, D\}$.
For the link diagram shown in Figure~\ref{fig:DLone}, $c^1_2(D_L)$ consists of two crossings. Going along the orientation from the base point of the first component, the image of $p$ of the first crossing is $b_-$ and the second one is $a_+$.    
    
\begin{figure}
    \includegraphics[width=0.5\linewidth]{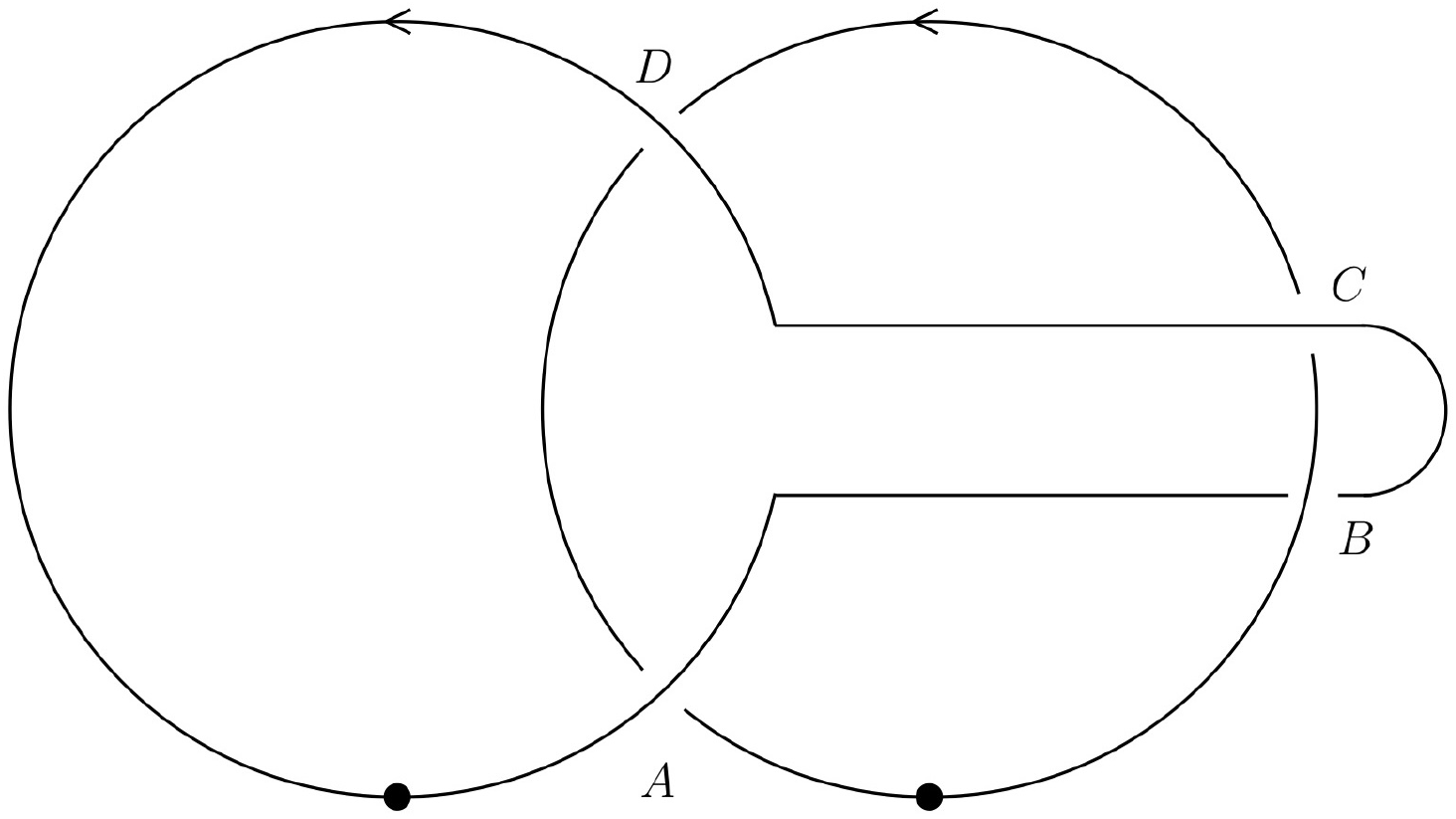}
    \caption{An ordered oriented link of two-components, each of which has a base point.
    We take the left component to be the first component.  
    }
\label{fig:Hopf}
\end{figure}
\end{ex} 
Let $c_1,\dots, c_{\ell}$ be the ordered crossings in $c^1_2(D_L)$.  
Then we define $f(D_L)$ by   
\begin{equation}\label{eq:bracket}
f(D_L) \coloneq p(c_1)p(c_2) \dots p(c_{\ell}) \in \pi(\alpha_*). 
\end{equation} 
\begin{ex}
For $L$ as in Figure~\ref{fig:Hopf}, reading  along the orientation of the first component (the left one) from the base point  gives the sequence of crossings $A, B, C, D$, and $c^1_2(D_L) =\{A, C, D\}$.  Under the mapping $p$, we have $p(A)=b_-$, $p(C)=b_-$, and $p(D)=a_+$.  Hence, 
$f(D_L)= b_- b_- a_+ $.  
\end{ex}
\begin{lem}\label{lem:GDF2c}
Given a $D_L$,  the image $f(D_L)$ does not change  under \tI~and~\tIII~ moves.    
\end{lem}
\begin{proof}
First, since $f(D_L)$ does not respond to self-intersections, it is invariant under the first and the third Reidemeister moves within a single component.  
Second, any Reidemeister move involving three components does not appear in two-component links.  Thus, what remains to be shown is that, among the third Reidemeister moves involving two components, we focus on the word consisting of a concatenation formed by reading the projections of $c^1_2 (D_L)$ between distinct components, from the base point of the first component.  

However, it is clear to see that the third Reidemeister move $\threeA$, $\threeB$, $\threeC$, $\threeD$, $\threeE$, $\threeF$, $\threeG$, or $\threeH$ does not change the crossing types $a_+, a_-, b_+, b_-$ between distinct components.  If the reader is familiar with Gauss diagrams, one will immediately understand this from 
\[
\includegraphics[width=0.45\linewidth]{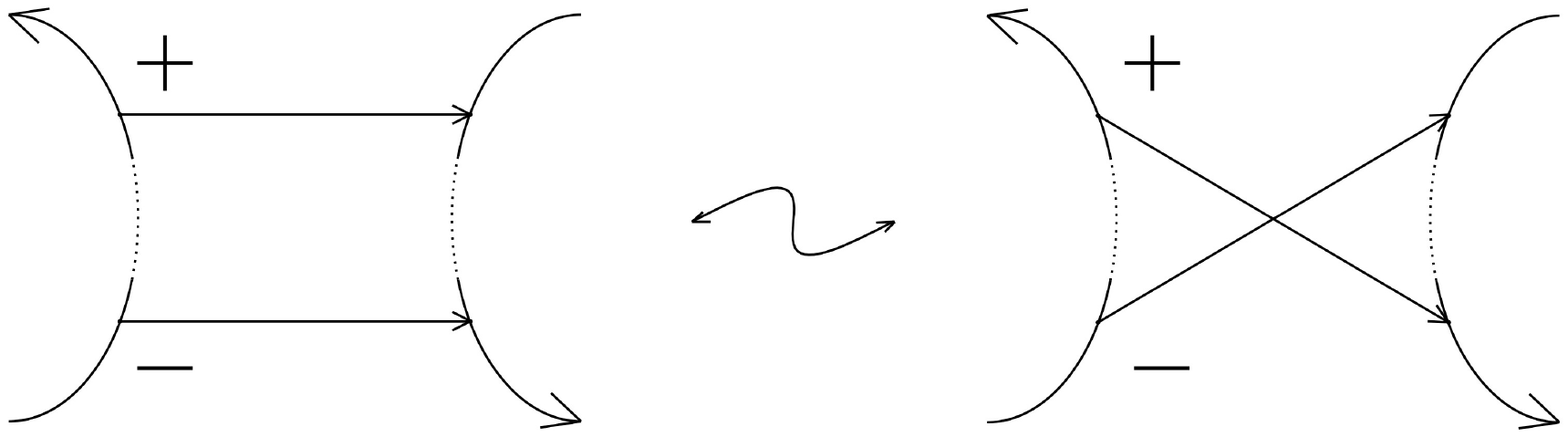}.
\]
Hence, the word does not change under \tIII~between two components, which implies the claim. 
\end{proof}

Every $(2, 2)$-tangle diagram will be a link diagram with $S^1$-boundary including four marked points, where some strands may intersect transversely.    
A simple closure  connects the  four marked endpoints in a fixed way to obtain a $2$-component link diagram, where each component is equipped with exactly one base point.  In this paper, we always take this simple closure so that the result is a $2$-component link.  The procedure defines an injective map from the set of $(2, 2)$-tangles to the set of such $2$-component links  with two base points.  In the above, we defined the invariant $f$ of these based two-component links, thereby obtaining invariants of $(2,2)$-tangles. 

\begin{rem}\label{rem:localGS}
The approach of Polyak \cite{Polyak2010} implicitly uses the principle that if a set is non-generating set of tangles, then it is non-generating set.   The argument in this section follows the same local-to-global logic explicitly.   
Concretely, Polyak constructed functions based on certain colorings, and we note that, following his convention, a set of Reidemeister moves is called a \emph{generating set} only when it generates all moves \emph{locally}.  Therefore, it is sufficient to show non-generating results using based and ordered two-component links or tangles.   
\end{rem}

\begin{figure}[h]
    \includegraphics[width=0.35\linewidth]{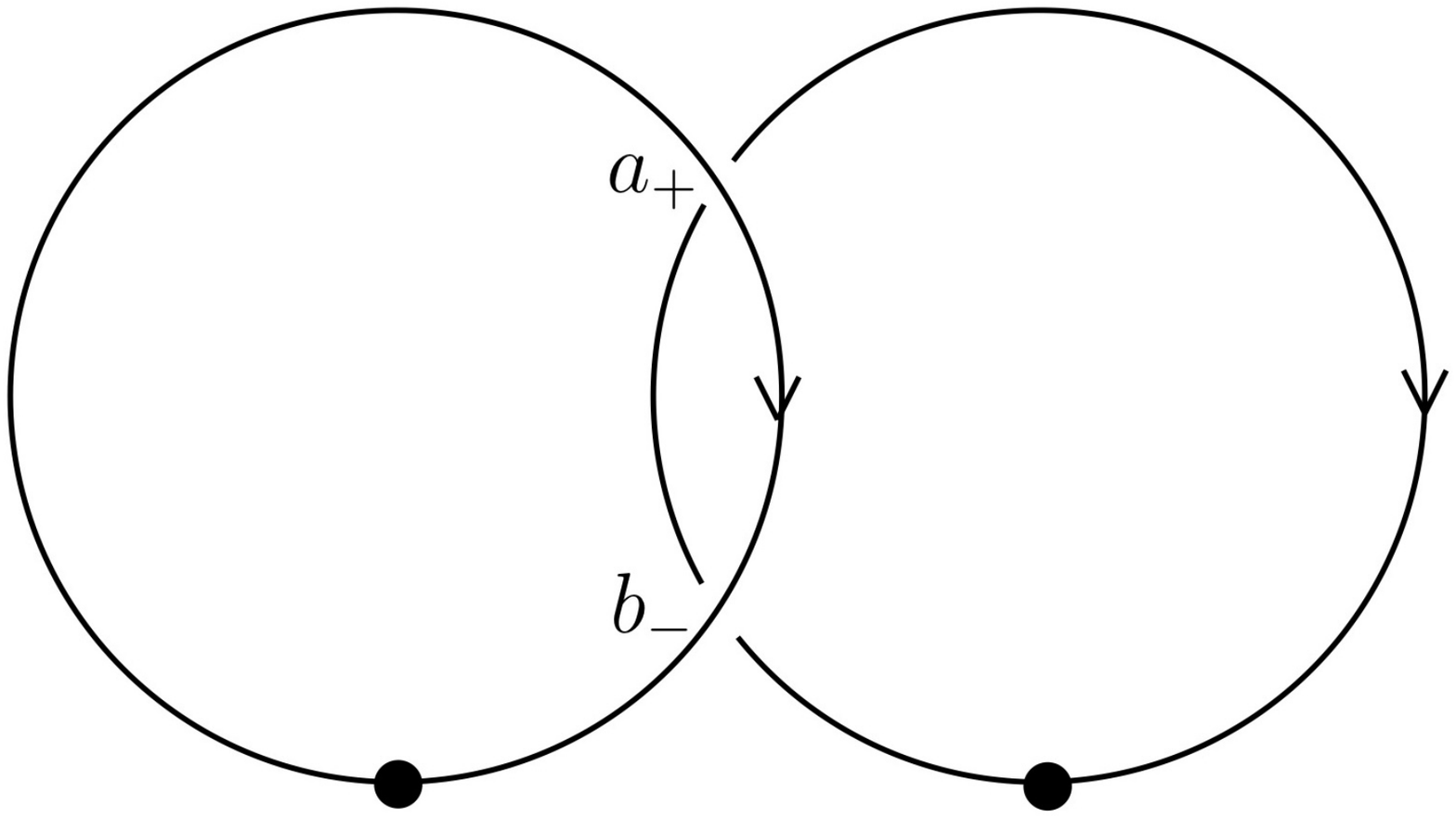}
    \caption{The link diagram $D_{L_1}$.  The left component represents the first component.}
    \label{fig:DLone}
\end{figure}

\begin{prop}\label{2cprop}
There exists a trivial two-component link, with a base point on each component, such that the number of crossings cannot be reduced to zero using all  of $\R I$, $\R \III$, and the move $\twoC$ alone. Here, the Reidemeister moves are allowed to be applied away from the basepoint.    
\end{prop}
\begin{proof}
Under our assumption, among moves of \tII, only $\twoC$ is allowed. 
Therefore, by Lemma~\ref{lem:GDF2c}, $\twoC$ is the only move that can change the value of $f$.
Further, $\twoC$ merely creates or annihilates the length-two block
\[
U:=  b_- a_+.  
\]
Assume for contradiction that the set $\R I \cup \{ \twoC \} \cup \R \III $ is a generating set.   Denote this assumption by ($\spadesuit$).  Under the assumption ($\spadesuit$), for the base-pointed trivial link diagram $D_{L_1}$ as in Figure~\ref{fig:DLone}, we would have $f(D_{L_1})=\emptyset$.  Denote this formula by ($\star$).

Fix a base-pointed trivial link diagram $D_{L_1}$ with exactly two crossings, and call them the \emph{original pair}.  The word obtained from the projections of two letters corresponding to the original pair is 
\[
a_+ b_- . 
\]

Then we apply elements of $\R I \cup \{ \twoC \} \cup \R \III $ in an arbitrary sequence.  Since only $\twoC$ affects $f$, and it does so solely by inserting or erasing copies of $U$, the image in the target of $f$ can differ from the original block $a_+ b_-$ only by placing $U$ to the left or right of $a_+ b_-$, or by inserting $U$ between the two original letters.  As a result, up to further insertions of $U$, the possibilities we must consider are:
\begin{enumerate}
\item[(1)] $U(a_+ b_-)$, 
\item[(2)] $(a_+ b_-)U$, 
\item[(3)] $\cdots a_+ U \cdots b_- \cdots$,
\item[(4)] $\cdots a_+ \cdots U  b_- \cdots$.
\end{enumerate}
where ``$\cdots$'' denotes a word generated by subsequent applications of $\twoC$.   

In cases (1) and (2), the letters adjacent across the boundary with the original pair align as $aa$ or $bb$; therefore, no cancellation toward the empty word occurs  there, and the value of $f$ does not become zero.  

It remains to treat (3) and (4), where a copy of $U= b_- a_+$ is placed between the two letters of the original pair.  The key is that the $m$ newborn pairs obtains sequence $(a_+ b_-)^{m+1}$ for a positive integer $m$. 
Hence each letter ($a_+$ or $b_-$) of the original pair cannot cancel against a letter belonging to the newly created $U$ via a $\twoC$ that adds or removes $b_- a_+$.  Consequently, at least $a_+$ ($=$ the first letter) or $b_-$ ($=$ the last letter) belonging to the original pair survives, thus $f(D_{L_1}) \neq \emptyset$.   

Therefore, no matter how many copies of $U$ are inserted or removed, the value of $f$ never becomes zero.  By induction on the number of applications  Reidemeister moves in $\R I \cup \{ \twoC \} \cup \R \III $, we conclude that $f(D_{L_1}) \neq \emptyset$, which contradicts ($\star$) under assumption ($\spadesuit$).  This contradiction implies the statement of this proposition.  
\end{proof}
%%NewInvFor2D%%
%In order to Proposition~\ref{2dprop}, we prepare Definition~\ref{2dCrossingSet}.  
%\begin{defn}\label{2dCrossingSet}
%Let $D_L$ be a two-component link diagram with an ordering of its components and a base point on each component.   Let $c^2_1(D_L)$ denote the totally ordered set of crossings where the under-path belongs to the first component and the over-path belongs to the second component, with the order given as $c_1, c_2, \dots, c_{\ell}$ by following the orientation of the first component starting from its base point.  
%\end{defn}
%Let $c_1,\dots, c_{\ell}$ be the ordered crossings in $c^2_1(D_L)$.  
%Then we define $g(D_L)$ by   
%\begin{equation}\label{eq:bracket}
%g(D_L) \coloneq p(c_1)p(c_2) \dots p(c_{\ell}) \in \pi(\alpha_*). 
%\end{equation}
%%%
\begin{prop}\label{2dprop}
There exists a trivial two-component link, with a base point on each component, such that the number of crossings cannot be reduced to zero using all  of $\R I$, $\R \III$, and the move $\twoD$ alone. Here, the Reidemeister moves are allowed to be applied away from the basepoint.    
\end{prop}

\begin{proof}
The following transformation rules allow us to obtain the proof of Proposition~\ref{2dprop} from Proposition~\ref{2cprop}.   

\begin{itemize}
\item The link diagram: $D_{L_1}$ (Figure~\ref{fig:DLone}) $\to$ $D_{L_2}$ (Figure~\ref{fig:DLtwo})
\item The allowed move of \tII ~: \twoC $\to$ \twoD
\item The original pair: $a_+b- \to b_-a+$ 
\item The block created by applying allowed move : $U$ $\to$ $V$
\end{itemize}
The details are left for the readers.
\end{proof}

\begin{figure}[h]
    \includegraphics[width=0.36\linewidth]{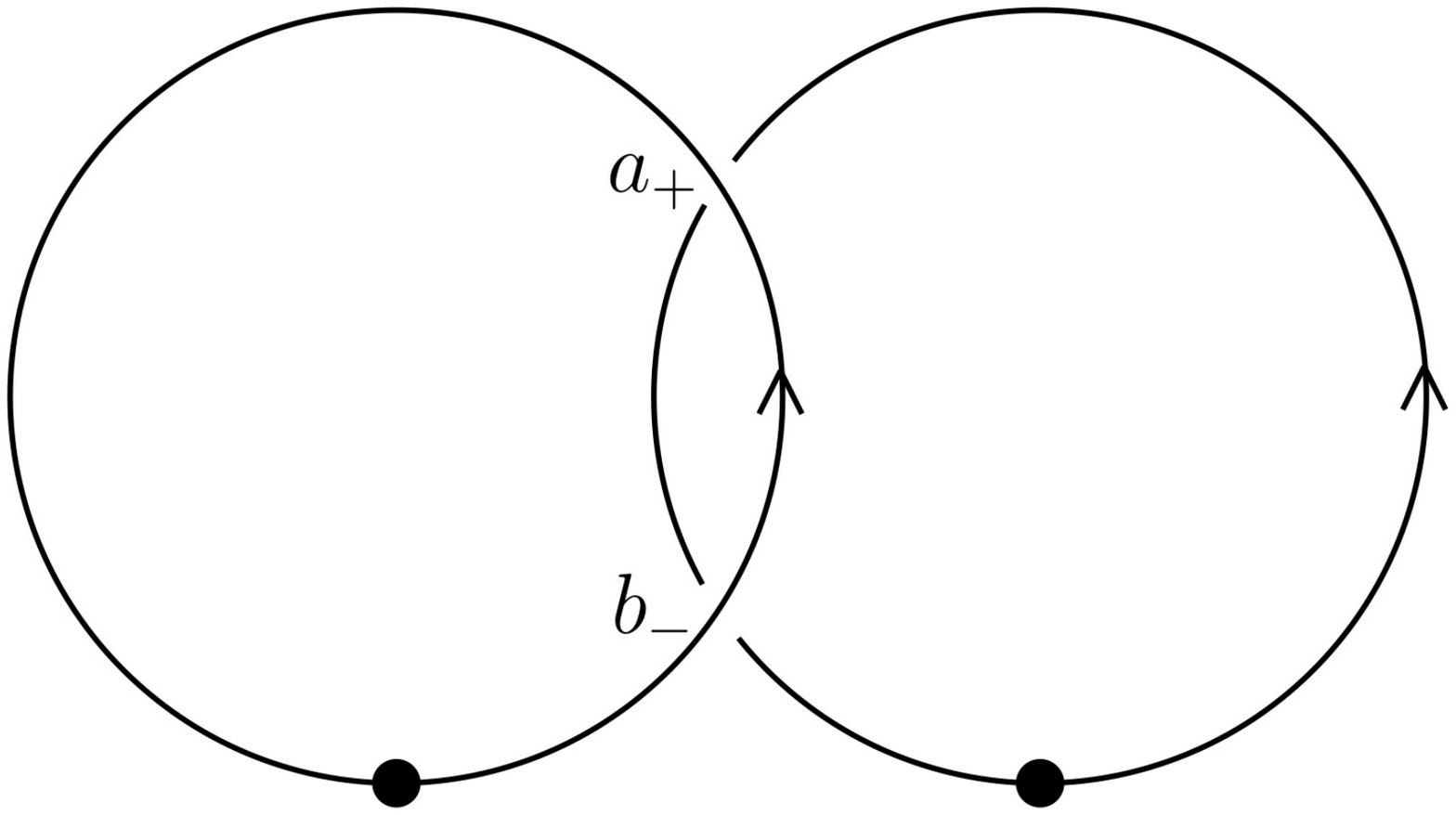}
    \caption{The link diagram $D_{L_2}$.  The left component represents the first component.}
    \label{fig:DLtwo}
\end{figure}

Based on Remark~\ref{rem:localGS}, Propositions~\ref{2cprop} and \ref{2dprop} imply  Corollary~\ref{cor:2c}.
\begin{cor}\label{cor:2c}
Neither of the sets 
$
\R I \cup  \{\twoC\} \cup \R \III
$
nor  
$
\R I \cup \{\twoD \} \cup \R \III
$
is a generating set.  
\end{cor}
%Proposition~\ref{2cprop} implies Corollary~\ref{cor:2c}.

\noindent({\bf Proof of Theorem~\ref{prop:cdAlone}})
Corollaries~\ref{cor:2c} immediately the claim.  
\hfill$\Box$

\begin{defn}
    We call the operation on a crossing of a link diagram $D$, as illustrated in the following figure,  \emph{smoothing}. A link diagram $D^{\prime}$ in which all crossings of $D$ are smoothed is said to be \emph{smoothed}.
    \begin{center}
        \includegraphics[width=0.3\linewidth]{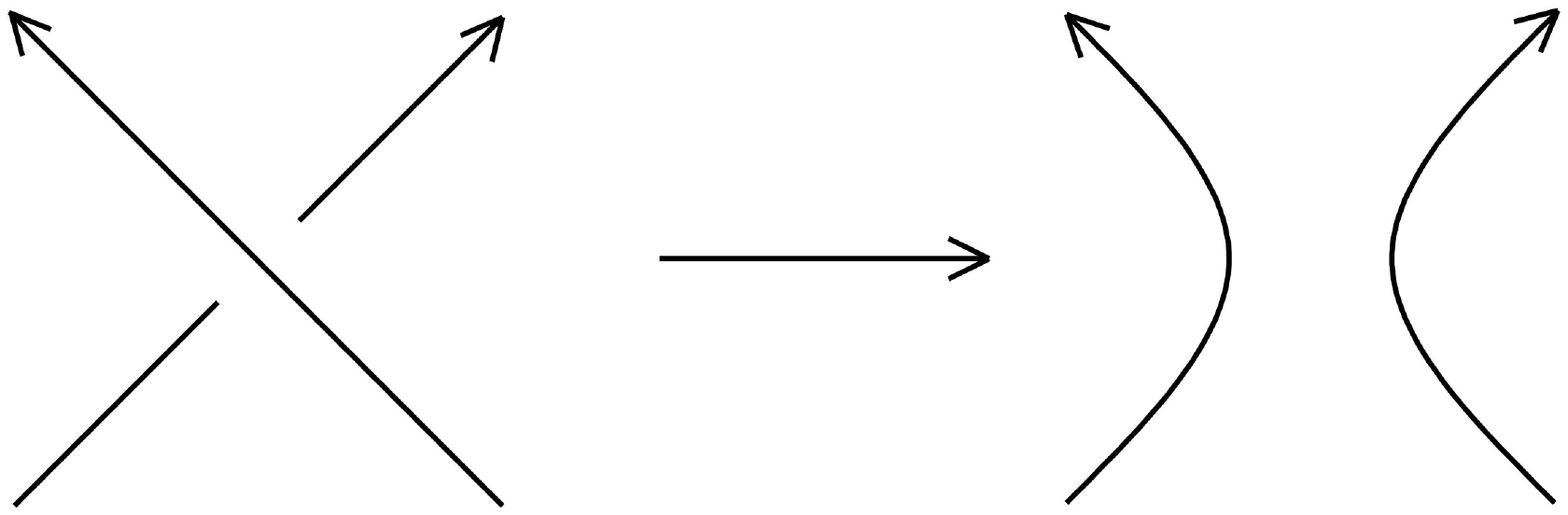}, 
        \includegraphics[width=0.3\linewidth]{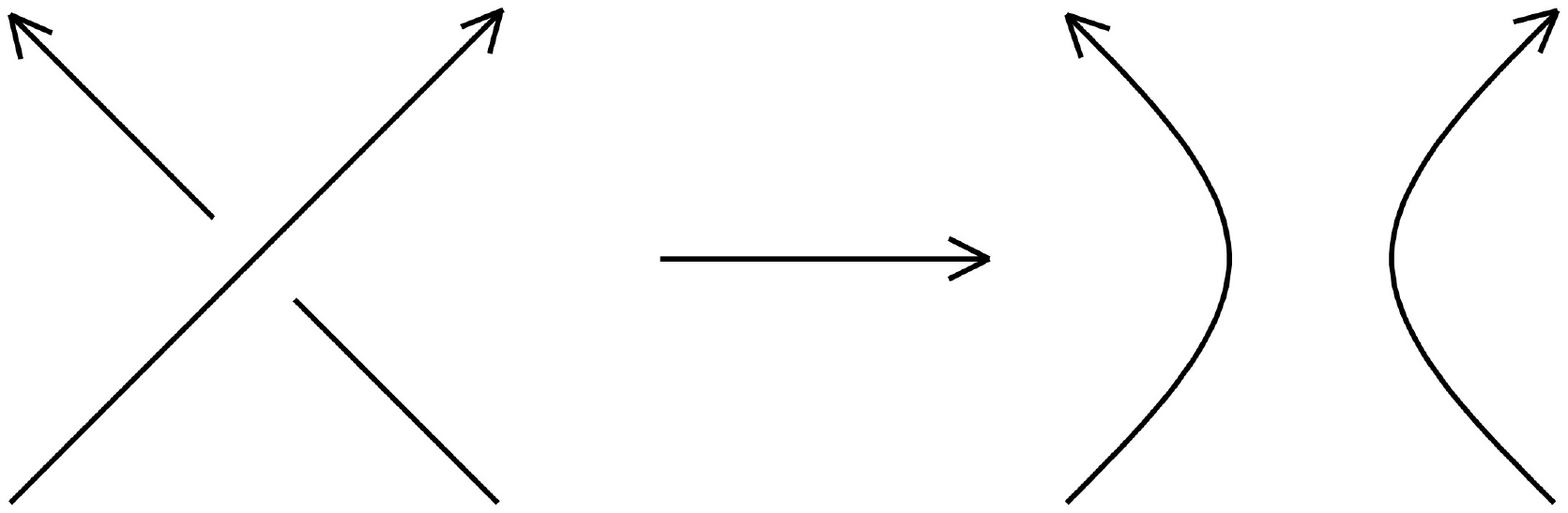}
    \end{center}
\end{defn}

\begin{prop}\label{braid&2cd}
For $\ast \in \{b, c, d, e, f, g\}$, the set $
\{\twoA, \twoB, \threeAS \} \cup CP 
$ 
does not form a generating set.  
\end{prop}
\begin{proof}
When a given link diagram $D$ is smoothed, let $C^-$ and $C^+$ denote the number of clockwise and counterclockwise oriented circles, respectively.  The smoothed diagram does not change under moves $\twoA, \twoB, \threeAS$ up to plane isotopy. But the move $\twoC$ (resp. $\twoD$) can change the number $C^+$ (resp. $C^-$). Let $w$ be the writhe number of $D$.  
For any compatible pair, we observe that
\begin{enumerate}
\item \label{Case1b1d} Moves $\oneB$ and $\oneD$ preserve $C^-$,
\item \label{Case1a1c} Moves $\oneA$ and $\oneC$ preserve $C^+$, 
\item \label{Case1a1b} Moves $\oneA$ and  $\oneB$ does not change the number $C^+ + C^- - w$,
\item \label{Case1c1d}  Case~$\oneC, \oneD$ does not change the number $C^+ + C^- + w$.
\end{enumerate}
Then we consider each case.  
\begin{itemize}
\item Assume that $\{\twoA, \twoB, \threeAS \} \cup \{ \oneB, \oneD \}$ were a generating set. This means that any move in $\R$ does not change the number of $C^-$. But it contradicts the fact that $\twoD$ can change $C^-$.  
\item Assume that $\{\twoA, \twoB, \threeAS \} \cup \{ \oneA, \oneC \}$ were a generating set. This means that any move in $\R$ does not change the number of $C^+$.  But it contradicts the fact that $\twoC$ can change $C^+$.  
\item Assume that $\{\twoA, \twoB, \threeAS \} \cup \{ \oneA, \oneB \}$ were a generating set.  This means that any move in $\R$ does not change the number of $C^+ + C^- - w$. But it contradicts the fact that both $\twoC$ and $\twoD$ can change $C^+ + C^- - w$. 
\item Assume that $\{\twoA, \twoB, \threeAS \} \cup \{ \oneC, \oneD \}$ were a generating set. This means that any move in $\R$ does not change the number of $C^+ + C^- + w$. But it contradicts the fact that both $\twoC$ and $\twoD$ can change $C^+ + C^- + w$. 
\end{itemize}
Hence, the set $\{\twoA, \twoB, \threeAS \} \cup CP$ does not form a generating set.  
\end{proof}

\noindent({\bf Proof of Theorem~\ref{Extend-of-Polyak}})

\noindent(\emph{Proof of (1)}.)

\noindent{(If part)}
If $\ast \in \{c, d, e, f \}$, the claim is immediately obtained by Corollary~\ref{Cor:3*Gen}.  
%By Lemma~\ref{Generating1},$\R I \prec S$, and by Lemma~\ref{Generating3}, either $\threeA$ or $\threeH \prec \R I \cup \{\twoC, \twoD, \threeAS\}$.  Then Lemma~\ref{Generating3}, we obtain $\threeB$ and $\threeG$.  
Hence, the remaining cases are $\threeB$ and $\threeG$.  The argument of each case is given below.  
\begin{itemize}
\item[($\threeB$)] 
    Let $S = \{ \threeB, \twoC, \twoD \} \cup CP$.  
    By Proposition~\ref{RI-gene}, 
    \[ \R I \prec S .\]
    By Lemma~\ref{Generating3}(9), 
    \[ \threeA \prec S .\]
    By Lemma~\ref{Generating3}(14), 
    \[ \threeF \prec S .\]
    By Lemma~\ref{Generating2}(5)(6), 
    \[ \twoA \prec S , \twoB \prec S .\] 
    Hence $S$ generates $\R I \cup \R \II \cup \{\threeB, \threeA,\threeF\}$. 
    By Proposition~\ref{3ast-all2},
    \[ \R \III \prec S .\]
    $S$ is a generating set.
\item[($\threeG$)] 
    Let $S = \{ \threeG, \twoC, \twoD \} \cup CP$.  
    By Proposition~\ref{RI-gene}, 
    \[ \R I \prec S .\]
    By Lemma~\ref{Generating3}(16), 
    \[ \threeH \prec S .\]
    By Lemma~\ref{Generating3}(13), 
    \[ \threeE \prec S .\]
    By Lemma~\ref{Generating2}(5)(6), 
    \[ \twoA \prec S , \twoB \prec S .\] 
    Hence $S$ generates $\R I \cup \R \II \cup \{\threeG, \threeH,\threeE \}$. 
    By Proposition~\ref{3ast-all2},
    \[ \R \III \prec S .\]
    $S$ is a generating set.
\end{itemize}
\noindent{(Only if part)}
Proposition~\ref{braid&2cd} implies that the set of $\{\twoA, \twoB, \threeAS \} \cup CP$ is not a generating set.  Therefore, if $S$ is a generating set, it must include either $\twoC$ or $\twoD$ as \tII.  More precisely, $S$ must contain either $\{\twoC, \threeAS\} \cup CP$ or $\{\twoD, \threeAS\} \cup CP$ as a subset.  By Theorem~\ref{prop:cdAlone}, there is no generating set whose the subset of the \tII~ moves equals $\{\twoC \}$ or $\{\twoD \}$.  Hence, $S$ must contain $\{ \twoC, \twoD, \threeAS \} \cup CP$.  It implies for any $CP$, $S$ includes $\{ \twoC, \twoD, \threeAS \} \cup CP$, which completes the proof.  

\noindent(\emph{Proof of (2)}.)
We treat the cases $\threeA$ and $\threeH$ separately.
\begin{itemize}
    \item[($\threeA$)]
        Let $S = \{ \threeA, \twoC, \twoD \} \cup CP$.
        By Proposition~\ref{RI-gene}, 
        \[ \R I \prec S .\]
        By Lemma~\ref{Generating3}(11), 
        \[ \threeC \prec S .\]
        By Lemma~\ref{Generating2}(5)(6), 
        \[ \twoA \prec S, \twoB \prec S .\] 
        Hence $S$ generates $\R I \cup \R \II \cup \{\threeA,\threeC\}$.
        
        By Proposition~\ref{3ast-all2},
        \[ \R \III \prec S .\]
        $S$ is a generating set.
    \item[($\threeH$)]
        Let $S = \{ \threeH, \twoC, \twoD \} \cup CP$.
        By Proposition~\ref{RI-gene}, 
        \[ \R I \prec S .\]
        By Lemma~\ref{Generating3}(12), 
        \[ \threeD \prec S .\]
        By Lemma~\ref{Generating2}(5)(6), 
        \[ \twoA \prec S , \twoB \prec S .\] 
        Hence $S$ generates $\R I \cup \R \II \cup \{\threeH,\threeD\}$.
        
        By Proposition~\ref{3ast-all2},
        \[ \R \III \prec S .\]
        $S$ is a generating set.
\end{itemize}
\hfill$\Box$

\section{Tables of generating sets}
%Below is a list of $4$-elements and $5$-elements collections of Reidemeister moves. It is intended to aid in determining whether a given collection is a generating set or a minimal generating set.

\begin{rem}\label{remhomotopy}
In \cite[Lemma~2.23] {CaprauScott2022}, Caprau and Scott conclude that the following four sets 
\[\{ \threeA, \twoA, \oneC, \oneD\}, 
\{ \threeA, \twoB, \oneC, \oneD\}, 
\{ \threeH, \twoA, \oneA, \oneB\}, 
\{ \threeH, \twoB, \oneA, \oneB \}
\]
are not generating sets, based on the inability to realize certain moves within a single movie move. 
We respect this conclusion and acknowledge its validity within their framework.  
Our observations examined all five known movie movies exhaustively (see Figure~\ref{fig:MM}).  
For example, for the set $\{ \threeA, \twoA, \oneC, \oneD\}$ we verified the following limitations:
    \begin{itemize}
        \item MI $($\tI$)$     \quad : By Lemma~\ref{Generating1}, in the absence of $\twoC$ and $\twoD$, generation of \tI~moves cannot be guaranteed.
        \item M\II $($\tII$)$  \quad : By Lemma~\ref{Generating2}, the generation of $\twoB$ cannot be guaranteed in the absence of $\threeC, \threeD, \threeE$ and $\threeF$; moreover, the generation of $\twoC$ and $\twoD$ cannot be guaranteed $\threeH$.
        \item M\III   \qquad\quad\quad \;        :In M\III, since the $2$-gons share a common double point, there seems to be no indication that $\{\twoA, \twoB\}$ produces $\{\twoC, \twoD\}$ and vise versa. 
        \item MI\!V $($\tIII$)$ : By Lemma~\ref{Generating3}, again due to the absence of $\twoC$ and $\twoD$, generation of \tIII~moves cannot be ensured.
        \item MV \qquad\qquad\; : In MV, starting from $\threeA$ only or $\threeH$ only, we do not see any way to add a new $3$-move.   It is elementary to show them, and we leave the details to  the reader.  
    \end{itemize}
    These observations suggest that the five known movie moves do not suffice to generate the missing moves. However, this does not constitute a complete proof of non-generation, since infinitely many loops in the $1$-skeleton of higher homotopies (arising from singularities of arbitrary codimension) remain unchecked.  Hence, while the result in  \cite[Lemma~2.23]{CaprauScott2022} demonstrates non-generation within a specific movie framework, the broader question of non-generation in the full homotopical sense remains unsolved.  We believe that developing new invariants will be essential to settle this question.  \footnote{This  remark had been shared with Professor Caprau prior to submission, in the spirit of constructive dialogue.}
\end{rem}

\begin{figure}
\begin{picture}(400,600)
\put(0,600){\LARGE{MI}}
        \put(0,400){\includegraphics[width=0.39\linewidth]{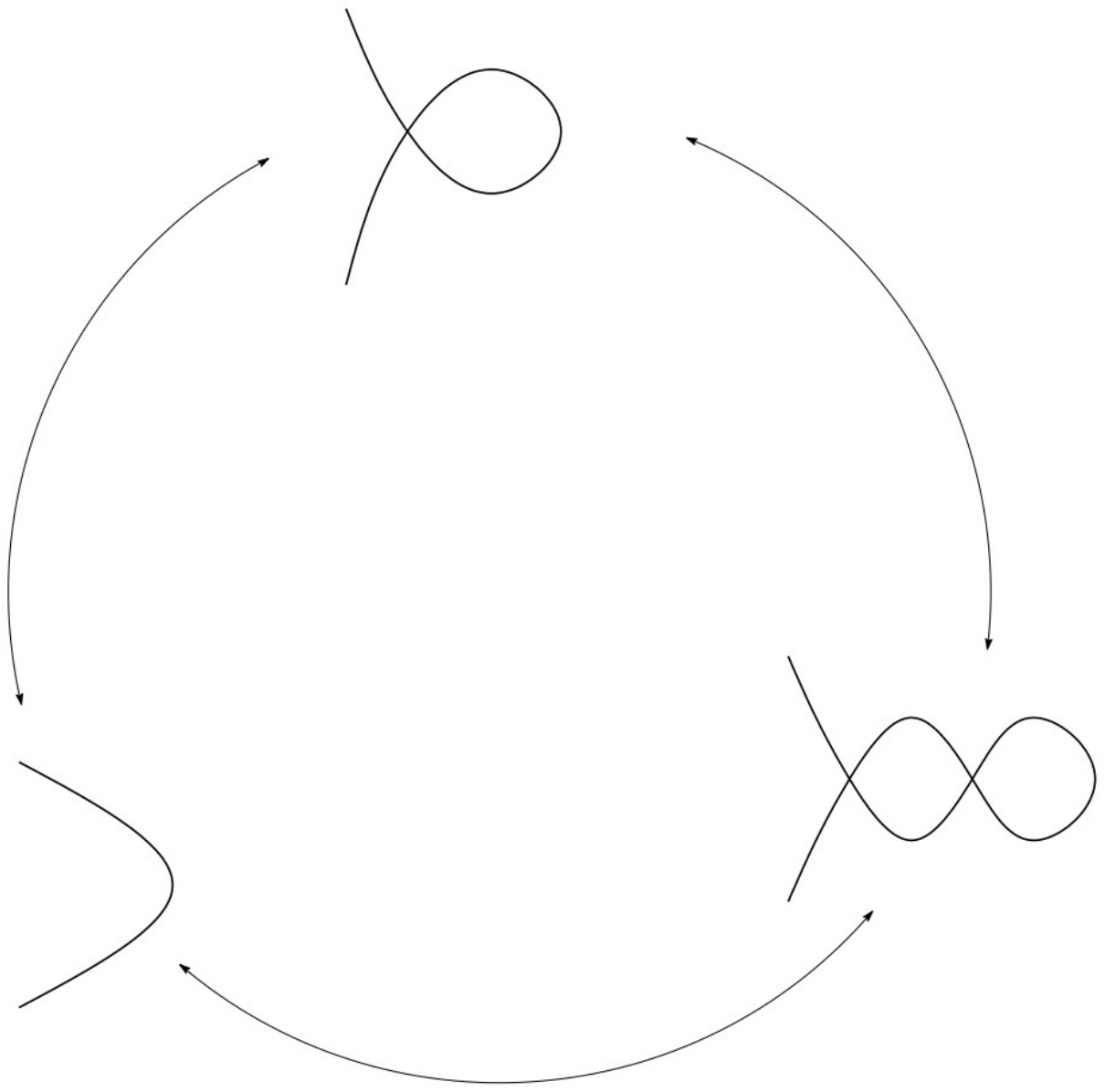}} 
\put(200,600){\LARGE{M\II}}
        \put(190,420){\includegraphics[width=0.6\linewidth]{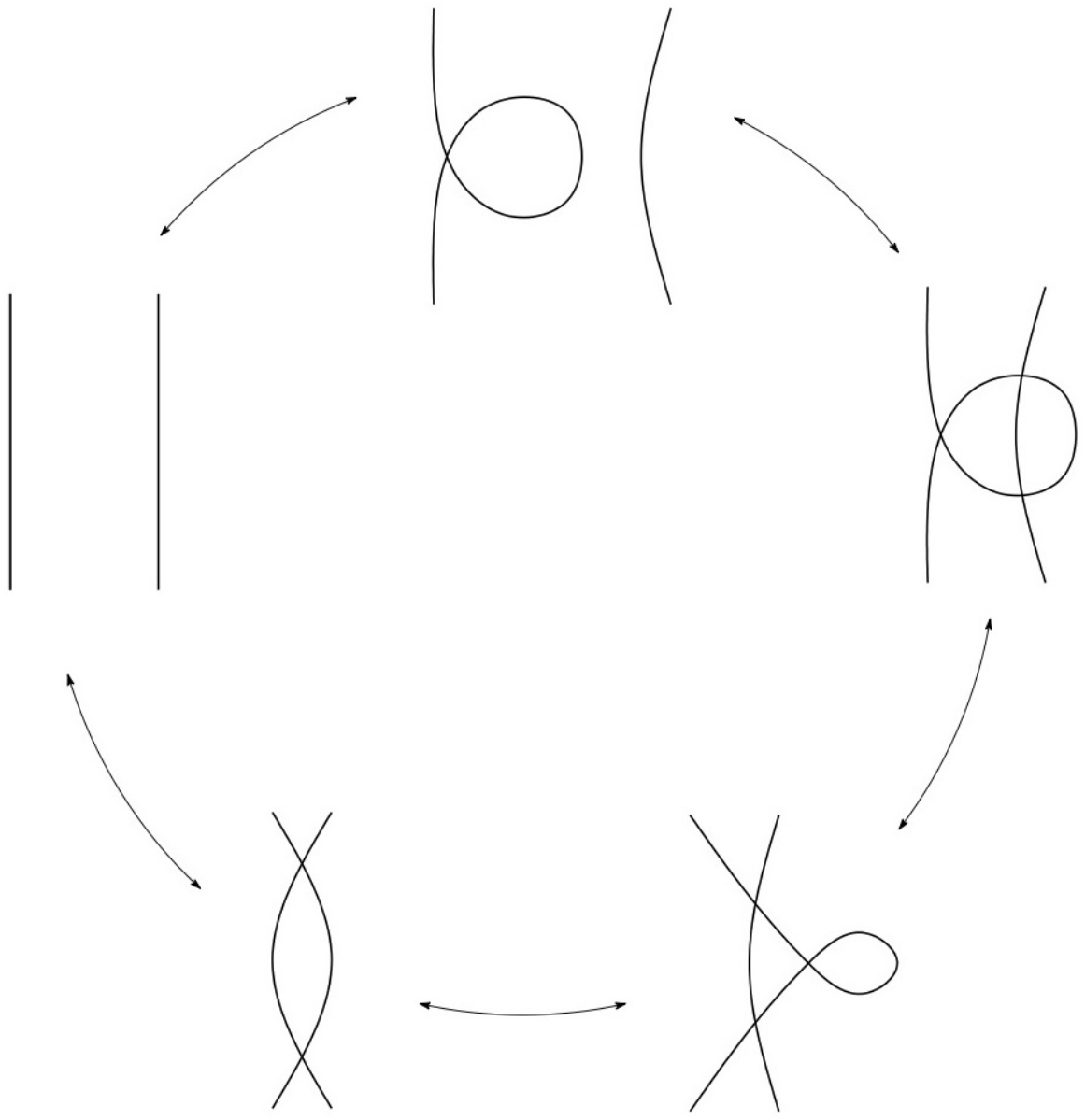}}
\put(0,390){\LARGE{M\III}}
        \put(-30,215){\includegraphics[width=0.5\linewidth]{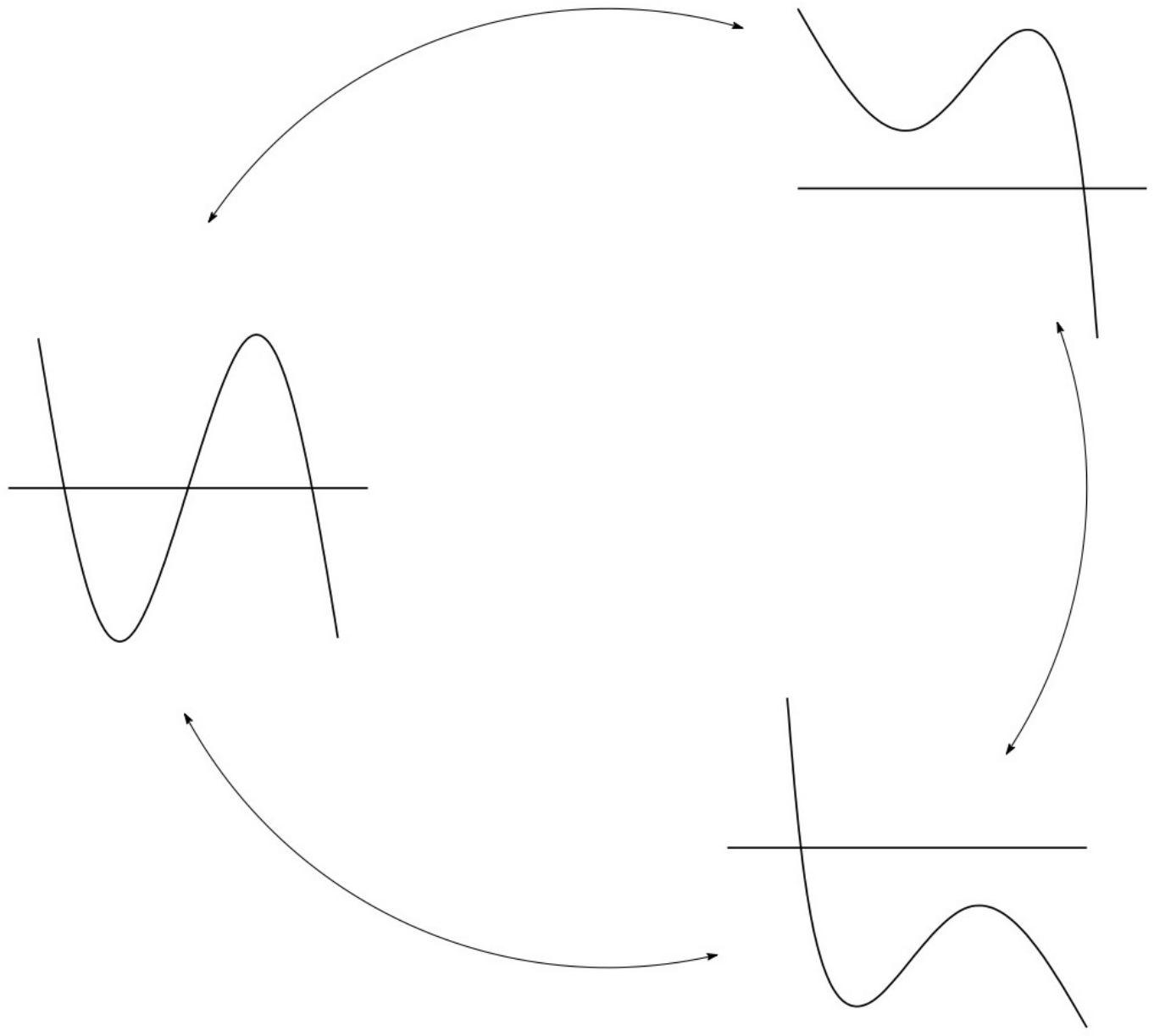}}
\put(200,390){\LARGE{MI\!V}}
        \put(185,195){\includegraphics[width=0.62\linewidth]{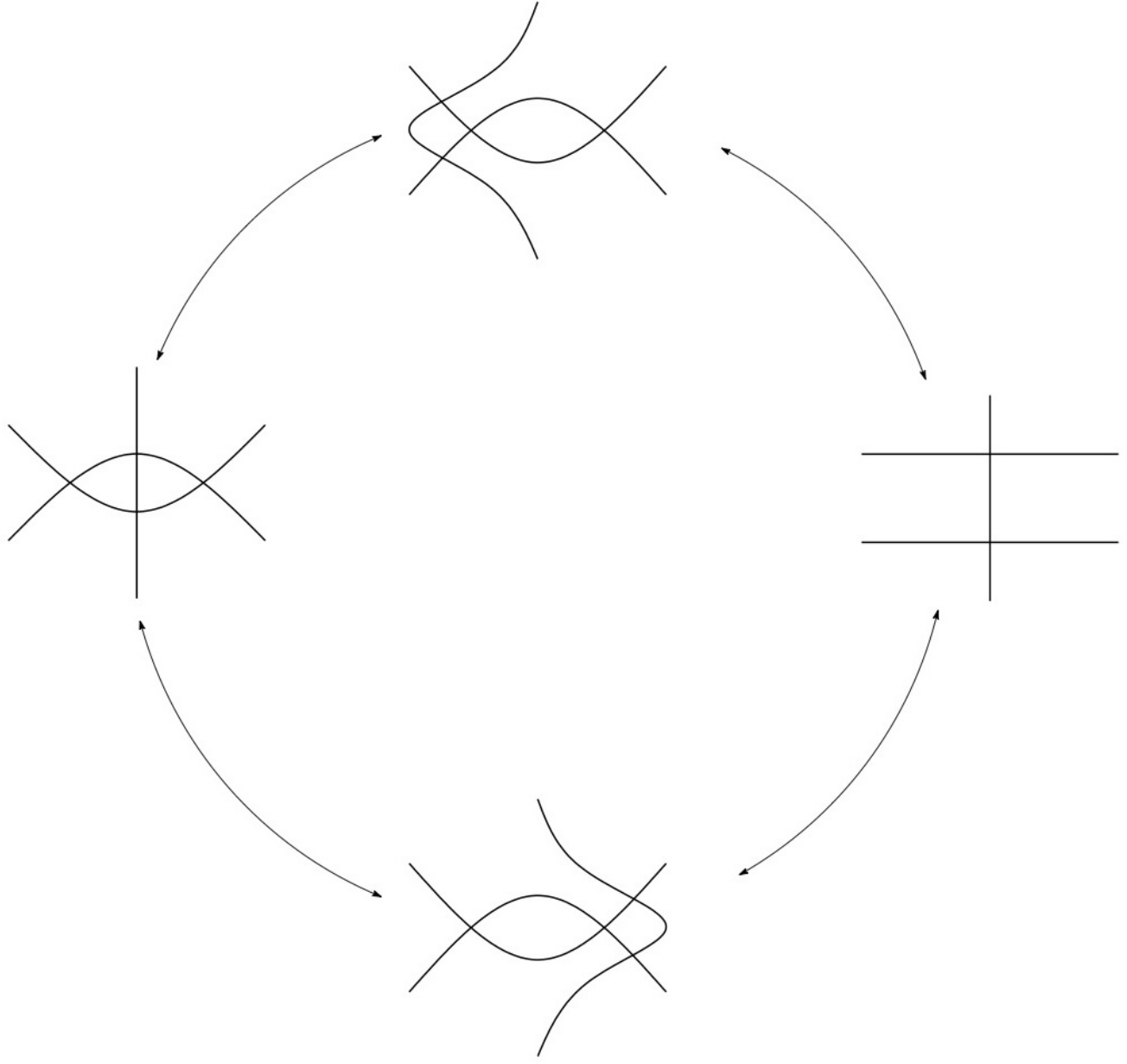}}
\put(70,170){\LARGE{MV}}
        \put(85,0){\includegraphics[width=0.6\linewidth]{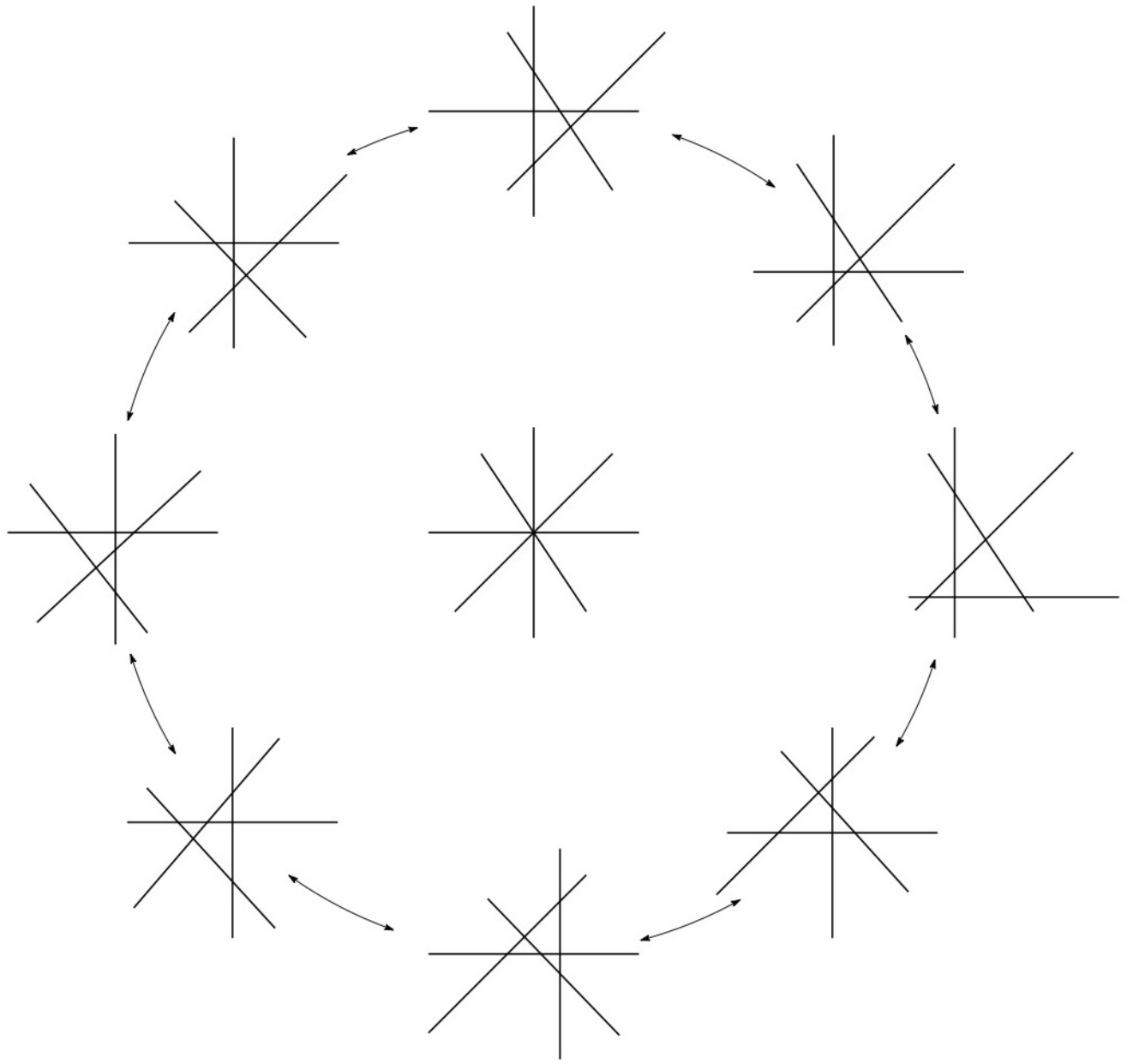}}
        \end{picture}
\caption{MI--MV.}\label{fig:MM}
\end{figure}

The symbol (M)GS stands for a (minimal) generating set.  The symbol $\times$ means that the corresponding set is not a  generating set.   
The symbol “?” indicates that it is not known whether the set is a generating set or not.

\begin{table}
\caption{Sets of Reidemeister moves containing $\threeA$ or $\threeH$ and their status. }\label{table:classificationAH}
\begin{tabular}{|c|c|c|c|} 
    \hline
        For $ \threeA$ (non-braid-type)       &          & For $\threeH$ (non-braid-type)             &               \\
    \hline
   $\{\threeA, \twoA, \oneA, \oneB \}       $ & MGS      & $\{\threeH, \twoA, \oneA, \oneB \}$        &  ?      \\
   $\{\threeA, \twoA, \oneA, \oneC \}       $ & MGS      & $\{\threeH, \twoA, \oneA, \oneC \}$        & MGS     \\
   $\{\threeA, \twoA, \oneB, \oneD \}       $ & MGS      & $\{\threeH, \twoA, \oneB, \oneD \}$        & MGS     \\
   $\{\threeA, \twoA, \oneC, \oneD \}       $ &  ?       & $\{\threeH, \twoA, \oneC, \oneD \}$        & MGS     \\
   $\{\threeA, \twoB, \oneA, \oneB \}       $ & MGS      & $\{\threeH, \twoB, \oneA, \oneB \}$        &  ?      \\
   $\{\threeA, \twoB, \oneA, \oneC \}       $ & MGS      & $\{\threeH, \twoB, \oneA, \oneC \}$        & MGS     \\
   $\{\threeA, \twoB, \oneB, \oneD \}       $ & MGS      & $\{\threeH, \twoB, \oneB, \oneD \}$        & MGS     \\
   $\{\threeA, \twoB, \oneC, \oneD \}       $ &  ?       & $\{\threeH, \twoB, \oneC, \oneD \}$        & MGS     \\
   $\{\threeA, \twoA, \twoB, \oneA, \oneB \}$ & GS       & $\{\threeH, \twoA, \twoB, \oneA, \oneB \}$ & ?        \\
   $\{\threeA, \twoA, \twoB, \oneA, \oneC \}$ & GS       & $\{\threeH, \twoA, \twoB, \oneA, \oneC \}$ & GS       \\
   $\{\threeA, \twoA, \twoB, \oneA, \oneD \}$ & $\times$ & $\{\threeH, \twoA, \twoB, \oneA, \oneD \}$ & $\times$ \\
   $\{\threeA, \twoA, \twoB, \oneB, \oneC \}$ & $\times$ & $\{\threeH, \twoA, \twoB, \oneB, \oneC \}$ & $\times$ \\
   $\{\threeA, \twoA, \twoB, \oneB, \oneD \}$ & GS       & $\{\threeH, \twoA, \twoB, \oneB, \oneD \}$ & GS       \\
   $\{\threeA, \twoA, \twoB, \oneC, \oneD \}$ & ?        & $\{\threeH, \twoA, \twoB, \oneC, \oneD \}$ & GS       \\
   $\{\threeA, \twoA, \twoC, \oneA, \oneB \}$ & GS       & $\{\threeH, \twoA, \twoC, \oneA, \oneB \}$ & ?        \\
   $\{\threeA, \twoA, \twoC, \oneA, \oneC \}$ & GS       & $\{\threeH, \twoA, \twoC, \oneA, \oneC \}$ & GS       \\
   $\{\threeA, \twoA, \twoC, \oneA, \oneD \}$ & $\times$ & $\{\threeH, \twoA, \twoC, \oneA, \oneD \}$ & $\times$ \\
   $\{\threeA, \twoA, \twoC, \oneB, \oneC \}$ & $\times$ & $\{\threeH, \twoA, \twoC, \oneB, \oneC \}$ & $\times$ \\
   $\{\threeA, \twoA, \twoC, \oneB, \oneD \}$ & GS       & $\{\threeH, \twoA, \twoC, \oneB, \oneD \}$ & GS       \\
   $\{\threeA, \twoA, \twoC, \oneC, \oneD \}$ & ?        & $\{\threeH, \twoA, \twoC, \oneC, \oneD \}$ & GS       \\
   $\{\threeA, \twoA, \twoD, \oneA, \oneB \}$ & GS       & $\{\threeH, \twoA, \twoD, \oneA, \oneB \}$ & ?        \\
   $\{\threeA, \twoA, \twoD, \oneA, \oneC \}$ & GS       & $\{\threeH, \twoA, \twoD, \oneA, \oneC \}$ & GS       \\
   $\{\threeA, \twoA, \twoD, \oneA, \oneD \}$ & $\times$ & $\{\threeH, \twoA, \twoD, \oneA, \oneD \}$ & $\times$ \\
   $\{\threeA, \twoA, \twoD, \oneB, \oneC \}$ & $\times$ & $\{\threeH, \twoA, \twoD, \oneB, \oneC \}$ & $\times$ \\
   $\{\threeA, \twoA, \twoD, \oneB, \oneD \}$ & GS       & $\{\threeH, \twoA, \twoD, \oneB, \oneD \}$ & GS       \\
   $\{\threeA, \twoA, \twoD, \oneC, \oneD \}$ & ?        & $\{\threeH, \twoA, \twoD, \oneC, \oneD \}$ & GS       \\
   $\{\threeA, \twoB, \twoC, \oneA, \oneB \}$ & GS       & $\{\threeH, \twoB, \twoC, \oneA, \oneB \}$ & ?        \\
   $\{\threeA, \twoB, \twoC, \oneA, \oneC \}$ & GS       & $\{\threeH, \twoB, \twoC, \oneA, \oneC \}$ & GS       \\
   $\{\threeA, \twoB, \twoC, \oneA, \oneD \}$ & $\times$ & $\{\threeH, \twoB, \twoC, \oneA, \oneD \}$ & $\times$ \\
   $\{\threeA, \twoB, \twoC, \oneB, \oneC \}$ & $\times$ & $\{\threeH, \twoB, \twoC, \oneB, \oneC \}$ & $\times$ \\
   $\{\threeA, \twoB, \twoC, \oneB, \oneD \}$ & GS       & $\{\threeH, \twoB, \twoC, \oneB, \oneD \}$ & GS       \\
   $\{\threeA, \twoB, \twoC, \oneC, \oneD \}$ & ?        & $\{\threeH, \twoB, \twoC, \oneC, \oneD \}$ & GS       \\
   $\{\threeA, \twoB, \twoD, \oneA, \oneB \}$ & GS       & $\{\threeH, \twoB, \twoD, \oneA, \oneB \}$ & ?        \\
   $\{\threeA, \twoB, \twoD, \oneA, \oneC \}$ & GS       & $\{\threeH, \twoB, \twoD, \oneA, \oneC \}$ & GS       \\
   $\{\threeA, \twoB, \twoD, \oneA, \oneD \}$ & $\times$ & $\{\threeH, \twoB, \twoD, \oneA, \oneD \}$ & $\times$ \\
   $\{\threeA, \twoB, \twoD, \oneB, \oneC \}$ & $\times$ & $\{\threeH, \twoB, \twoD, \oneB, \oneC \}$ & $\times$ \\
   $\{\threeA, \twoB, \twoD, \oneB, \oneD \}$ & GS       & $\{\threeH, \twoB, \twoD, \oneB, \oneD \}$ & GS       \\
   $\{\threeA, \twoB, \twoD, \oneC, \oneD \}$ & ?        & $\{\threeH, \twoB, \twoD, \oneC, \oneD \}$ & GS       \\
   $\{\threeA, \twoC, \twoD, \oneA, \oneB \}$ & MGS      & $\{\threeH, \twoC, \twoD, \oneA, \oneB \}$ & MGS      \\
   $\{\threeA, \twoC, \twoD, \oneA, \oneC \}$ & MGS      & $\{\threeH, \twoC, \twoD, \oneA, \oneC \}$ & MGS      \\
   $\{\threeA, \twoC, \twoD, \oneA, \oneD \}$ & $\times$ & $\{\threeH, \twoC, \twoD, \oneA, \oneD \}$ & $\times$ \\
   $\{\threeA, \twoC, \twoD, \oneB, \oneC \}$ & $\times$ & $\{\threeH, \twoC, \twoD, \oneB, \oneC \}$ & $\times$ \\
   $\{\threeA, \twoC, \twoD, \oneB, \oneD \}$ & MGS      & $\{\threeH, \twoC, \twoD, \oneB, \oneD \}$ & MGS      \\
   $\{\threeA, \twoC, \twoD, \oneC, \oneD \}$ & MGS      & $\{\threeH, \twoC, \twoD, \oneC, \oneD \}$ & MGS      \\
\hline
\end{tabular}
\end{table}

\clearpage

\begin{table}
\caption{Sets of Reidemeister moves containing $\threeB$, $\threeC$, $\threeD$ and their status.}\label{table:classificationBCD}
\begin{tabular}{|c|c|c|c|c|c|} 
\hline
    For $\threeB$ (braid-type) &  & For $\threeC$ (braid-type) &  & For $\threeD$ (braid-type)  &    \\  
\hline
$\{\threeB, \twoA, \twoB, \oneA, \oneB \}$ & $\times$ &
            $\{\threeC, \twoA, \twoB, \oneA, \oneB \}$ & $\times$ &
                    $\{\threeD, \twoA, \twoB, \oneA, \oneB \}$ & $\times$ \\
   $\{\threeB, \twoA, \twoB, \oneA, \oneC \}$ & $\times$ &
            $\{\threeC, \twoA, \twoB, \oneA, \oneC \}$ & $\times$ &
                    $\{\threeD, \twoA, \twoB, \oneA, \oneC \}$ & $\times$ \\
   $\{\threeB, \twoA, \twoB, \oneA, \oneD \}$ & $\times$ &
            $\{\threeC, \twoA, \twoB, \oneA, \oneD \}$ & $\times$ &
                    $\{\threeD, \twoA, \twoB, \oneA, \oneD \}$ & $\times$ \\
   $\{\threeB, \twoA, \twoB, \oneB, \oneC \}$ & $\times$ &
            $\{\threeC, \twoA, \twoB, \oneB, \oneC \}$ & $\times$ &
                    $\{\threeD, \twoA, \twoB, \oneB, \oneC \}$ & $\times$ \\
   $\{\threeB, \twoA, \twoB, \oneB, \oneD \}$ & $\times$ &
            $\{\threeC, \twoA, \twoB, \oneB, \oneD \}$ & $\times$ &
                    $\{\threeD, \twoA, \twoB, \oneB, \oneD \}$ & $\times$ \\
   $\{\threeB, \twoA, \twoB, \oneC, \oneD \}$ & $\times$ &
            $\{\threeC, \twoA, \twoB, \oneC, \oneD \}$ & $\times$ &
                    $\{\threeD, \twoA, \twoB, \oneC, \oneD \}$ & $\times$ \\
   $\{\threeB, \twoA, \twoC, \oneA, \oneB \}$ & $\times$ &
            $\{\threeC, \twoA, \twoC, \oneA, \oneB \}$ & $\times$ &
                    $\{\threeD, \twoA, \twoC, \oneA, \oneB \}$ & $\times$ \\
   $\{\threeB, \twoA, \twoC, \oneA, \oneC \}$ & $\times$ &
            $\{\threeC, \twoA, \twoC, \oneA, \oneC \}$ & $\times$ &
                    $\{\threeD, \twoA, \twoC, \oneA, \oneC \}$ & $\times$ \\
   $\{\threeB, \twoA, \twoC, \oneA, \oneD \}$ & $\times$ &
            $\{\threeC, \twoA, \twoC, \oneA, \oneD \}$ & $\times$ &
                    $\{\threeD, \twoA, \twoC, \oneA, \oneD \}$ & $\times$ \\   
   $\{\threeB, \twoA, \twoC, \oneB, \oneC \}$ & $\times$ &
            $\{\threeC, \twoA, \twoC, \oneB, \oneC \}$ & $\times$ &
                    $\{\threeD, \twoA, \twoC, \oneB, \oneC \}$ & $\times$ \\
   $\{\threeB, \twoA, \twoC, \oneB, \oneD \}$ & $\times$ &
            $\{\threeC, \twoA, \twoC, \oneB, \oneD \}$ & $\times$ &
                    $\{\threeD, \twoA, \twoC, \oneB, \oneD \}$ & $\times$ \\
   $\{\threeB, \twoA, \twoC, \oneC, \oneD \}$ & $\times$ &
            $\{\threeC, \twoA, \twoC, \oneC, \oneD \}$ & $\times$ &
                    $\{\threeD, \twoA, \twoC, \oneC, \oneD \}$ & $\times$ \\
   $\{\threeB, \twoA, \twoD, \oneA, \oneB \}$ & $\times$ &
            $\{\threeC, \twoA, \twoD, \oneA, \oneB \}$ & $\times$ &
                    $\{\threeD, \twoA, \twoD, \oneA, \oneB \}$ & $\times$ \\
   $\{\threeB, \twoA, \twoD, \oneA, \oneC \}$ & $\times$ &
            $\{\threeC, \twoA, \twoD, \oneA, \oneC \}$ & $\times$ &
                    $\{\threeD, \twoA, \twoD, \oneA, \oneC \}$ & $\times$ \\
   $\{\threeB, \twoA, \twoD, \oneA, \oneD \}$ & $\times$ &
            $\{\threeC, \twoA, \twoD, \oneA, \oneD \}$ & $\times$ &
                    $\{\threeD, \twoA, \twoD, \oneA, \oneD \}$ & $\times$ \\
   $\{\threeB, \twoA, \twoD, \oneB, \oneC \}$ & $\times$ &
            $\{\threeC, \twoA, \twoD, \oneB, \oneC \}$ & $\times$ &
                    $\{\threeD, \twoA, \twoD, \oneB, \oneC \}$ & $\times$ \\
   $\{\threeB, \twoA, \twoD, \oneB, \oneD \}$ & $\times$ &
            $\{\threeC, \twoA, \twoD, \oneB, \oneD \}$ & $\times$ &
                    $\{\threeD, \twoA, \twoD, \oneB, \oneD \}$ & $\times$ \\
   $\{\threeB, \twoA, \twoD, \oneC, \oneD \}$ & $\times$ &
            $\{\threeC, \twoA, \twoD, \oneC, \oneD \}$ & $\times$ &
                    $\{\threeD, \twoA, \twoD, \oneC, \oneD \}$ & $\times$ \\
   $\{\threeB, \twoB, \twoC, \oneA, \oneB \}$ & $\times$ &
            $\{\threeC, \twoB, \twoC, \oneA, \oneB \}$ & $\times$ &
                    $\{\threeD, \twoB, \twoC, \oneA, \oneB \}$ & $\times$ \\
   $\{\threeB, \twoB, \twoC, \oneA, \oneC \}$ & $\times$ &
            $\{\threeC, \twoB, \twoC, \oneA, \oneC \}$ & $\times$ &
                    $\{\threeD, \twoB, \twoC, \oneA, \oneC \}$ & $\times$ \\
   $\{\threeB, \twoB, \twoC, \oneA, \oneD \}$ & $\times$ &
            $\{\threeC, \twoB, \twoC, \oneA, \oneD \}$ & $\times$ &
                    $\{\threeD, \twoB, \twoC, \oneA, \oneD \}$ & $\times$ \\
   $\{\threeB, \twoB, \twoC, \oneB, \oneC \}$ & $\times$ &
            $\{\threeC, \twoB, \twoC, \oneB, \oneC \}$ & $\times$ &
                    $\{\threeD, \twoB, \twoC, \oneB, \oneC \}$ & $\times$ \\
   $\{\threeB, \twoB, \twoC, \oneB, \oneD \}$ & $\times$ &
            $\{\threeC, \twoB, \twoC, \oneB, \oneD \}$ & $\times$ &
                    $\{\threeD, \twoB, \twoC, \oneB, \oneD \}$ & $\times$ \\
   $\{\threeB, \twoB, \twoC, \oneC, \oneD \}$ & $\times$ &
            $\{\threeC, \twoB, \twoC, \oneC, \oneD \}$ & $\times$ &
                    $\{\threeD, \twoB, \twoC, \oneC, \oneD \}$ & $\times$ \\
   $\{\threeB, \twoB, \twoD, \oneA, \oneB \}$ & $\times$ &
            $\{\threeC, \twoB, \twoD, \oneA, \oneB \}$ & $\times$ &
                    $\{\threeD, \twoB, \twoD, \oneA, \oneB \}$ & $\times$ \\
   $\{\threeB, \twoB, \twoD, \oneA, \oneC \}$ & $\times$ &
            $\{\threeC, \twoB, \twoD, \oneA, \oneC \}$ & $\times$ &
                    $\{\threeD, \twoB, \twoD, \oneA, \oneC \}$ & $\times$ \\
   $\{\threeB, \twoB, \twoD, \oneA, \oneD \}$ & $\times$ &
            $\{\threeC, \twoB, \twoD, \oneA, \oneD \}$ & $\times$ &
                    $\{\threeD, \twoB, \twoD, \oneA, \oneD \}$ & $\times$ \\
   $\{\threeB, \twoB, \twoD, \oneB, \oneC \}$ & $\times$ &
            $\{\threeC, \twoB, \twoD, \oneB, \oneC \}$ & $\times$ &
                    $\{\threeD, \twoB, \twoD, \oneB, \oneC \}$ & $\times$ \\
   $\{\threeB, \twoB, \twoD, \oneB, \oneD \}$ & $\times$ &
            $\{\threeC, \twoB, \twoD, \oneB, \oneD \}$ & $\times$ &
                    $\{\threeD, \twoB, \twoD, \oneB, \oneD \}$ & $\times$ \\ 
   $\{\threeB, \twoB, \twoD, \oneC, \oneD \}$ & $\times$ &
            $\{\threeC, \twoB, \twoD, \oneC, \oneD \}$ & $\times$ &
                    $\{\threeD, \twoB, \twoD, \oneC, \oneD \}$ & $\times$ \\
   $\{\threeB, \twoC, \twoD, \oneA, \oneB \}$ & MGS      &
            $\{\threeC, \twoC, \twoD, \oneA, \oneB \}$ & MGS      &
                    $\{\threeD, \twoC, \twoD, \oneA, \oneB \}$ & MGS      \\
   $\{\threeB, \twoC, \twoD, \oneA, \oneC \}$ & MGS      &
            $\{\threeC, \twoC, \twoD, \oneA, \oneC \}$ & MGS      &
                    $\{\threeD, \twoC, \twoD, \oneA, \oneC \}$ & MGS      \\
   $\{\threeB, \twoC, \twoD, \oneA, \oneD \}$ & $\times$ &
            $\{\threeC, \twoC, \twoD, \oneA, \oneD \}$ & $\times$ &
                    $\{\threeD, \twoC, \twoD, \oneA, \oneD \}$ & $\times$ \\
   $\{\threeB, \twoC, \twoD, \oneB, \oneC \}$ & $\times$ &
            $\{\threeC, \twoC, \twoD, \oneB, \oneC \}$ & $\times$ &
                    $\{\threeD, \twoC, \twoD, \oneB, \oneC \}$ & $\times$ \\
   $\{\threeB, \twoC, \twoD, \oneB, \oneD \}$ & MGS      &
            $\{\threeC, \twoC, \twoD, \oneB, \oneD \}$ & MGS      &
                    $\{\threeD, \twoC, \twoD, \oneB, \oneD \}$ & MGS      \\
   $\{\threeB, \twoC, \twoD, \oneC, \oneD \}$ & MGS      &
            $\{\threeC, \twoC, \twoD, \oneC, \oneD \}$ & MGS      &
                    $\{\threeD, \twoC, \twoD, \oneC, \oneD \}$ & MGS      \\

\hline
\end{tabular}
\end{table}

\begin{table}
\caption{Sets of Reidemeister moves containing $\threeE$, $\threeF$, $\threeG$ and their status.}\label{table:classificationDEF}
\begin{tabular}{|c|c|c|c|c|c|} 
\hline
    For $\threeE$ (braid-type) &  & For $\threeF$ (braid-type) &  & For $\threeG$ (braid-type)  &    \\  
\hline

$\{\threeE, \twoA, \twoB, \oneA, \oneB \}$ & $\times$ &
            $\{\threeF, \twoA, \twoB, \oneA, \oneB \}$ & $\times$ &
                    $\{\threeG, \twoA, \twoB, \oneA, \oneB \}$ & $\times$ \\
   $\{\threeE, \twoA, \twoB, \oneA, \oneC \}$ & $\times$ &
            $\{\threeF, \twoA, \twoB, \oneA, \oneC \}$ & $\times$ &
                    $\{\threeG, \twoA, \twoB, \oneA, \oneC \}$ & $\times$ \\
   $\{\threeE, \twoA, \twoB, \oneA, \oneD \}$ & $\times$ &
            $\{\threeF, \twoA, \twoB, \oneA, \oneD \}$ & $\times$ &
                    $\{\threeG, \twoA, \twoB, \oneA, \oneD \}$ & $\times$ \\
   $\{\threeE, \twoA, \twoB, \oneB, \oneC \}$ & $\times$ &
            $\{\threeF, \twoA, \twoB, \oneB, \oneC \}$ & $\times$ &
                    $\{\threeG, \twoA, \twoB, \oneB, \oneC \}$ & $\times$ \\
   $\{\threeE, \twoA, \twoB, \oneB, \oneD \}$ & $\times$ &
            $\{\threeF, \twoA, \twoB, \oneB, \oneD \}$ & $\times$ &
                    $\{\threeG, \twoA, \twoB, \oneB, \oneD \}$ & $\times$ \\
   $\{\threeE, \twoA, \twoB, \oneC, \oneD \}$ & $\times$ &
            $\{\threeF, \twoA, \twoB, \oneC, \oneD \}$ & $\times$ &
                    $\{\threeG, \twoA, \twoB, \oneC, \oneD \}$ & $\times$ \\
   $\{\threeE, \twoA, \twoC, \oneA, \oneB \}$ & $\times$ &
            $\{\threeF, \twoA, \twoC, \oneA, \oneB \}$ & $\times$ &
                    $\{\threeG, \twoA, \twoC, \oneA, \oneB \}$ & $\times$ \\
   $\{\threeE, \twoA, \twoC, \oneA, \oneC \}$ & $\times$ &
            $\{\threeF, \twoA, \twoC, \oneA, \oneC \}$ & $\times$ &
                    $\{\threeG, \twoA, \twoC, \oneA, \oneC \}$ & $\times$ \\
   $\{\threeE, \twoA, \twoC, \oneA, \oneD \}$ & $\times$ &
            $\{\threeF, \twoA, \twoC, \oneA, \oneD \}$ & $\times$ &
                    $\{\threeG, \twoA, \twoC, \oneA, \oneD \}$ & $\times$ \\   
   $\{\threeE, \twoA, \twoC, \oneB, \oneC \}$ & $\times$ &
            $\{\threeF, \twoA, \twoC, \oneB, \oneC \}$ & $\times$ &
                    $\{\threeG, \twoA, \twoC, \oneB, \oneC \}$ & $\times$ \\
   $\{\threeE, \twoA, \twoC, \oneB, \oneD \}$ & $\times$ &
            $\{\threeF, \twoA, \twoC, \oneB, \oneD \}$ & $\times$ &
                    $\{\threeG, \twoA, \twoC, \oneB, \oneD \}$ & $\times$ \\
   $\{\threeE, \twoA, \twoC, \oneC, \oneD \}$ & $\times$ &
            $\{\threeF, \twoA, \twoC, \oneC, \oneD \}$ & $\times$ &
                    $\{\threeG, \twoA, \twoC, \oneC, \oneD \}$ & $\times$ \\
   $\{\threeE, \twoA, \twoD, \oneA, \oneB \}$ & $\times$ &
            $\{\threeF, \twoA, \twoD, \oneA, \oneB \}$ & $\times$ &
                    $\{\threeG, \twoA, \twoD, \oneA, \oneB \}$ & $\times$ \\
   $\{\threeE, \twoA, \twoD, \oneA, \oneC \}$ & $\times$ &
            $\{\threeF, \twoA, \twoD, \oneA, \oneC \}$ & $\times$ &
                    $\{\threeG, \twoA, \twoD, \oneA, \oneC \}$ & $\times$ \\
   $\{\threeE, \twoA, \twoD, \oneA, \oneD \}$ & $\times$ &
            $\{\threeF, \twoA, \twoD, \oneA, \oneD \}$ & $\times$ &
                    $\{\threeG, \twoA, \twoD, \oneA, \oneD \}$ & $\times$ \\
   $\{\threeE, \twoA, \twoD, \oneB, \oneC \}$ & $\times$ &
            $\{\threeF, \twoA, \twoD, \oneB, \oneC \}$ & $\times$ &
                    $\{\threeG, \twoA, \twoD, \oneB, \oneC \}$ & $\times$ \\
   $\{\threeE, \twoA, \twoD, \oneB, \oneD \}$ & $\times$ &
            $\{\threeF, \twoA, \twoD, \oneB, \oneD \}$ & $\times$ &
                    $\{\threeG, \twoA, \twoD, \oneB, \oneD \}$ & $\times$ \\
   $\{\threeE, \twoA, \twoD, \oneC, \oneD \}$ & $\times$ &
            $\{\threeF, \twoA, \twoD, \oneC, \oneD \}$ & $\times$ &
                    $\{\threeG, \twoA, \twoD, \oneC, \oneD \}$ & $\times$ \\
   $\{\threeE, \twoB, \twoC, \oneA, \oneB \}$ & $\times$ &
            $\{\threeF, \twoB, \twoC, \oneA, \oneB \}$ & $\times$ &
                    $\{\threeG, \twoB, \twoC, \oneA, \oneB \}$ & $\times$ \\
   $\{\threeE, \twoB, \twoC, \oneA, \oneC \}$ & $\times$ &
            $\{\threeF, \twoB, \twoC, \oneA, \oneC \}$ & $\times$ &
                    $\{\threeG, \twoB, \twoC, \oneA, \oneC \}$ & $\times$ \\
   $\{\threeE, \twoB, \twoC, \oneA, \oneD \}$ & $\times$ &
            $\{\threeF, \twoB, \twoC, \oneA, \oneD \}$ & $\times$ &
                    $\{\threeG, \twoB, \twoC, \oneA, \oneD \}$ & $\times$ \\
   $\{\threeE, \twoB, \twoC, \oneB, \oneC \}$ & $\times$ &
            $\{\threeF, \twoB, \twoC, \oneB, \oneC \}$ & $\times$ &
                    $\{\threeG, \twoB, \twoC, \oneB, \oneC \}$ & $\times$ \\
   $\{\threeE, \twoB, \twoC, \oneB, \oneD \}$ & $\times$ &
            $\{\threeF, \twoB, \twoC, \oneB, \oneD \}$ & $\times$ &
                    $\{\threeG, \twoB, \twoC, \oneB, \oneD \}$ & $\times$ \\
   $\{\threeE, \twoB, \twoC, \oneC, \oneD \}$ & $\times$ &
            $\{\threeF, \twoB, \twoC, \oneC, \oneD \}$ & $\times$ &
                    $\{\threeG, \twoB, \twoC, \oneC, \oneD \}$ & $\times$ \\
   $\{\threeE, \twoB, \twoD, \oneA, \oneB \}$ & $\times$ &
            $\{\threeF, \twoB, \twoD, \oneA, \oneB \}$ & $\times$ &
                    $\{\threeG, \twoB, \twoD, \oneA, \oneB \}$ & $\times$ \\
   $\{\threeE, \twoB, \twoD, \oneA, \oneC \}$ & $\times$ &
            $\{\threeF, \twoB, \twoD, \oneA, \oneC \}$ & $\times$ &
                    $\{\threeG, \twoB, \twoD, \oneA, \oneC \}$ & $\times$ \\
   $\{\threeE, \twoB, \twoD, \oneA, \oneD \}$ & $\times$ &
            $\{\threeF, \twoB, \twoD, \oneA, \oneD \}$ & $\times$ &
                    $\{\threeG, \twoB, \twoD, \oneA, \oneD \}$ & $\times$ \\
   $\{\threeE, \twoB, \twoD, \oneB, \oneC \}$ & $\times$ &
            $\{\threeF, \twoB, \twoD, \oneB, \oneC \}$ & $\times$ &
                    $\{\threeG, \twoB, \twoD, \oneB, \oneC \}$ & $\times$ \\
   $\{\threeE, \twoB, \twoD, \oneB, \oneD \}$ & $\times$ &
            $\{\threeF, \twoB, \twoD, \oneB, \oneD \}$ & $\times$ &
                    $\{\threeG, \twoB, \twoD, \oneB, \oneD \}$ & $\times$ \\ 
   $\{\threeE, \twoB, \twoD, \oneC, \oneD \}$ & $\times$ &
            $\{\threeF, \twoB, \twoD, \oneC, \oneD \}$ & $\times$ &
                    $\{\threeG, \twoB, \twoD, \oneC, \oneD \}$ & $\times$ \\
   $\{\threeE, \twoC, \twoD, \oneA, \oneB \}$ & MGS      &
            $\{\threeF, \twoC, \twoD, \oneA, \oneB \}$ & MGS      &
                    $\{\threeG, \twoC, \twoD, \oneA, \oneB \}$ & MGS      \\
   $\{\threeE, \twoC, \twoD, \oneA, \oneC \}$ & MGS      &
            $\{\threeF, \twoC, \twoD, \oneA, \oneC \}$ & MGS      &
                    $\{\threeG, \twoC, \twoD, \oneA, \oneC \}$ & MGS      \\
   $\{\threeE, \twoC, \twoD, \oneA, \oneD \}$ & $\times$ &
            $\{\threeF, \twoC, \twoD, \oneA, \oneD \}$ & $\times$ &
                    $\{\threeG, \twoC, \twoD, \oneA, \oneD \}$ & $\times$ \\
   $\{\threeE, \twoC, \twoD, \oneB, \oneC \}$ & $\times$ &
            $\{\threeF, \twoC, \twoD, \oneB, \oneC \}$ & $\times$ &
                    $\{\threeG, \twoC, \twoD, \oneB, \oneC \}$ & $\times$ \\
   $\{\threeE, \twoC, \twoD, \oneB, \oneD \}$ & MGS      &
            $\{\threeF, \twoC, \twoD, \oneB, \oneD \}$ & MGS      &
                    $\{\threeG, \twoC, \twoD, \oneB, \oneD \}$ & MGS      \\
   $\{\threeE, \twoC, \twoD, \oneC, \oneD \}$ & MGS      &
            $\{\threeF, \twoC, \twoD, \oneC, \oneD \}$ & MGS      &
                    $\{\threeG, \twoC, \twoD, \oneC, \oneD \}$ & MGS      \\
   
\hline
\end{tabular}
\end{table}

\clearpage

\section*{Acknowledgments}
The authors would like to thank Professor Keiichi Sakai for his valuable  comments.  
The work was partially  supported by KAKENHI number JP25K06999.  
This work was also partially supported by JST SPRING, Japan Grant Number JPMJSP2144 (Shinshu University).
\bibliographystyle{plain}
\bibliography{ListMGS}
\end{document}